\documentclass[a4paper]{amsart}

\usepackage{lmodern}
\usepackage{xr-hyper}
\usepackage[dvipsnames,svgnames,x11names,hyperref]{xcolor}
\usepackage[colorlinks,citecolor=Mahogany,linkcolor=Mahogany,urlcolor=Mahogany,filecolor=Mahogany]{hyperref}

\usepackage{amsmath,amstext,amssymb,mathrsfs,amscd,amsthm,indentfirst}
\usepackage{amsfonts}
\usepackage{mathtools}
\usepackage{stmaryrd}
\usepackage{bbm}
\usepackage{units}
\usepackage{microtype}
\usepackage{booktabs}
\usepackage[capitalise]{cleveref}
\usepackage{arydshln} %for dashed lines in tables

\usepackage{caption}
\usepackage{booktabs}
\usepackage{verbatim}
\usepackage{enumerate}

\numberwithin{equation}{section}

\newcommand{\cat}[1]{\mathsf{#1}}

\newcommand{\mr}[1]{{\rm #1}}
\newcommand{\bunit}{\mathbbm{1}}
\newcommand{\titsrel}[2]{\cT(#1\,\mr{rel}^0\,#2)}
\newcommand{\titsdualrel}[2]{\cT(#1\,\mr{rel}^+\,#2)}
\newcommand{\strel}[2]{\mr{St}(#1\,\mr{rel}^0\,#2)}
\newcommand{\trel}[2]{T(#1\,\mr{rel}^0\,#2)}
\newcommand{\tdualrel}[2]{T(#1\,\mr{rel}^+\,#2)}
\newcommand{\steonerel}[2]{\mr{St}^{E_1}(#1\,\mr{rel}^0\,#2)}
\usepackage{tikz,tikz-cd}

\newcommand{\CircNum}[1]{\ooalign{\hfil\raise .00ex\hbox{\scriptsize #1}\hfil\crcr\mathhexbox20D}}

\newcommand{\bF}{\mathbb{F}}

\newcommand{\bL}{\mathbb{L}}

\newcommand{\bN}{\mathbb{N}}

\newcommand{\bP}{\mathbb{P}}
\newcommand{\bQ}{\mathbb{Q}}

\newcommand{\bZ}{\mathbb{Z}}

\newcommand{\gA}{\bold{A}}

\newcommand{\gC}{\bold{C}}
\newcommand{\gD}{\bold{D}}
\newcommand{\gE}{\bold{E}}

\newcommand{\gN}{\bold{N}}

\newcommand{\gR}{\bold{R}}

\newcommand{\gT}{\bold{T}}

\newcommand{\cC}{\mathcal{C}}

\newcommand{\cL}{\mathcal{L}}

\newcommand{\cO}{\mathcal{O}}

\newcommand{\cS}{\mathcal{S}}
\newcommand{\cT}{\mathcal{T}}

\newcommand{\cX}{\mathcal{X}}
\newcommand{\cY}{\mathcal{Y}}

\newcommand{\bk}{\mathbbm{k}}
\newcommand{\sdash}{\,\cdot\,}

\newcommand\lra{\longrightarrow}

\newcommand{\hcoker}{/\!\!/}

\newcommand{\fS}{\mathfrak{S}}

\newcommand{\moddd}{/\!\!/}

\renewcommand{\epsilon}{\varepsilon}

\newcommand{\Sq}{\mathrm{Sq}}

\newcommand{\Alg}{\cat{Alg}}

\mathchardef\ordinarycolon\mathcode`\:
\mathcode`\:=\string"8000
\begingroup \catcode`\:=\active
  \gdef:{\mathrel{\mathop\ordinarycolon}}
\endgroup

\usepackage{amsthm}
\theoremstyle{plain}
\newtheorem{MainThm}{Theorem}
\newtheorem{MainCor}[MainThm]{Corollary}

\newtheorem{theorem}{Theorem}[section]

\newtheorem{proposition}[theorem]{Proposition}
\newtheorem{lemma}[theorem]{Lemma}

\newtheorem{corollary}[theorem]{Corollary}

\theoremstyle{definition}
\newtheorem{definition}[theorem]{Definition}

\newtheorem{notation}[theorem]{Notation}

\theoremstyle{remark}

\newtheorem{remark}[theorem]{Remark}
\newtheorem*{remark*}{Remark}

\title[$E_\infty$-cells and general linear groups of finite fields]{$E_\infty$-cells and general linear groups of finite fields \\[.5cm]
$E_{\infty}$-cellules et groupes généraux linéaires sur les corps finis}

\author{S{\o}ren Galatius}
\address{Department of Mathematics\\
	University of Copenhagen\\
	Universitetsparken 5 \\
	2100 K{\o}benhavn {\O} \\
	Denmark}
\email{galatius@math.ku.dk}

\author{Alexander Kupers}
\address{Department of Computer and Mathematical Sciences \\
	University of Toronto Scarborough \\
	1265 Military Trail \\
	Toronto, ON M1C 1A4 \\ 
	Canada}
\email{a.kupers@utoronto.ca}

\author{Oscar Randal-Williams}
\address{Centre for Mathematical Sciences\\
 	Wilberforce Road\\
 	Cambridge CB3 0WB\\
 	UK}
\email{o.randal-williams@dpmms.cam.ac.uk}

\date{\today}

\begin{document}
	
\begin{abstract}
We prove new homological stability results for general linear groups over finite fields. These results are obtained by constructing CW approximations to the classifying spaces of these groups, in the category of $E_\infty$-algebras, guided by computations of homology with coefficients in the $E_1$-split Steinberg module. \\

Nous prouvons de nouveaux résultats de stabilité homologique pour les groupes généraux linéaires sur les corps finis. Ces résultats sont obtenus en construisant des approximations cellulaires du classifiant de ces  groupes dans la catégorie des algèbres $E_{\infty}$, guidées par des calculs d’homologie à coefficients dans le module de Steinberg $E_1$-scindé.
\end{abstract}

\maketitle

\vspace{-1.5\baselineskip}

\section{Introduction}The homology of general linear groups over finite fields was studied by Quillen in his seminal paper on the $K$-theory of finite fields  \cite[Theorem 3]{quillenfinite}. One of his main results is the complete determination of the homology groups $H_d(\mr{GL}_n(\bF_q);\bF_\ell)$ where $\bF_q$ is a finite field with $q = p^r$ elements, $\mr{GL}_n(\bF_q)$ denotes the general linear group of the $n$-dimensional vector space $\bF_q^n$, and $\ell \neq p$.

When $\ell=p$ full information about these homology groups is only available in a stable range. Extending linear isomorphisms of $\bF_q^{n-1}$ to $\bF_q^n = \bF_q^{n-1} \oplus \bF_q$ by the identity on the second subspace induces an injective homomorphism 
\[s \colon \mr{GL}_{n-1}(\bF_q) \lra \mr{GL}_{n}(\bF_q)\]
called the \emph{stabilisation map}. Forming the colimit over such stabilisation maps, Quillen has proved \cite[Corollary 2]{quillenfinite} that 
\begin{equation}\label{eq:QuillenVanishing}
H_d(\mr{GL}_\infty(\bF_q);\bF_p) = 0 \text{ for } d > 0,
\end{equation}
which has implications for finite $n$ using homological stability. The best ho\-mo\-lo\-gi\-cal stability result available in the literature is due to Maazen, who proved that $H_d(\mr{GL}_n(\bF_q),\mr{GL}_{n-1}(\bF_q);\bF_p) = 0$ for $d < \frac{n}{2}$ \cite{maazenthesis}. It is an unpublished result of Quillen that for $q \neq 2$ these groups vanish in the larger range $d<n$. (A proof appears in his unpublished notebooks \cite[1974-I, p.~10]{quillennotes}. The first pages of this notebook were unfortunately bleached by sunlight, which makes it difficult to follow the argument. A recent paper of Sprehn and Wahl \cite{SprehnWahl} include an exposition of Quillen's argument, as well as a generalization to other classical groups.)

\subsection{Homological stability} In this paper we study the homology of general linear groups of finite fields of characteristic $p$, with coefficients in $\bF_p$, using $E_k$-cellular approximations.  We developed this technique in \cite{e2cellsIv3}, and in \cite{e2cellsII} we applied it to studying homology of mapping class groups. The case of infinite fields is treated in \cite{e2cellsIV}. For fields with more than two elements we obtain the following strengthening of Quillen's result (see also \cite[Theorem A]{SprehnWahl}).

\begin{MainThm}\label{thm:main-fp} 
If $q = p^r \neq 2$, then 
\[H_d(\mr{GL}_n(\bF_q),\mr{GL}_{n-1}(\bF_q);\bF_p) = 0 \text{ for } d < n+r(p-1)-2.\]
\end{MainThm}

The bound $n + r(p-1) - 2$ is linear of slope 1 as was Quillen's, but our offset $r(p-1)-2$ is strictly positive for $q > 4$.  Our methods also apply to the field $\bF_2$, where we obtain the following linear bound of slope $2/3$.

\begin{MainThm}\label{thm:main-f2} $H_{d}(\mr{GL}_{n}(\bF_2),\mr{GL}_{n-1}(\bF_2);\bF_2) = 0$ for $d < \frac{2(n-1)}{3}$.
\end{MainThm}

Combining this with Quillen's calculation \eqref{eq:QuillenVanishing} (or using \cref{cor.stabilisation-vanishing} when $q=2$), we obtain the following vanishing theorem.

\begin{MainCor}\label{cor:vanishing}
If $q = p^r \neq 2$, then $\widetilde{H}_d(\mr{GL}_n(\bF_q);\bF_p)=0$ for $d+1 < n+r(p-1)-1$. If $q = 2$, then $\widetilde{H}_d(\mr{GL}_n(\bF_2);\bF_2)=0$ for $d+1< \frac{2n}{3}$.
\end{MainCor}

\subsection{A question of Milgram and Priddy}
In \cite{milgram-priddy}, Milgram and Priddy constructed a cohomology class $\det_n \in H^n(M_{n,n}(\bF_2);\bF_2)$, where $M_{n,n}(\bF_2) \subset \mr{GL}_{2n}(\bF_2)$ is the subgroup of $(2n \times 2n)$-matrices of the form
\[\begin{bmatrix}
\mr{id}_n & * \\
0 & \mr{id}_n
\end{bmatrix},\]
which is invariant for the natural action of ${\mr{GL}_n(\bF_2) \times \mr{GL}_n(\bF_2)}$
as the subgroup of block-diagonal matrices. They suggested (see also \cite[Section 5]{priddy-problemsession}) that this class may be in the image under the restriction map
\begin{equation*}
  H^*(\mr{GL}_{2n}(\bF_2);\bF_2) \lra H^*(M_{n,n}(\bF_2);\bF_2)^{\mr{GL}_n(\bF_2) \times \mr{GL}_n(\bF_2)}.
\end{equation*}
By Maazen's result it would then be the lowest-degree such class.  They showed that
this is indeed the case for $\det_1$ and $\det_2$. However \cref{cor:vanishing} implies that $H_n(\mr{GL}_{2n}(\bF_2);\bF_2)=0$ for $n > 3$, so that $\det_n$ \emph{cannot} be in the image of the restriction map for $n>3$.  In \cref{sec:sharpnessf2} we combine our techniques with a recent computation of Szymik to prove that $\det_3$ does lie in the image of the restriction map. This completely answers the question of Milgram and Priddy.

\subsection{Homology of Steinberg modules} The reduced top integral homology of the Tits building associated to $\bF_q^n$ is equal to the Steinberg module $\mr{St}(\bF_q^n)$. It has an action of $\mr{GL}_n(\bF_q)$, and in the modular representation theory of $\mr{GL}_n(\bF_q)$ a central role is played by the $\bF_p[\mr{GL}_n(\bF_q)]$-module $\mr{St}(\bF_q^n) \otimes \bF_p$, as it is the only module which is both irreducible and projective. This implies that $H_*(\mr{GL}_n(\bF_q) ; \mr{St}(\bF_q^n) \otimes \bF_p)=0$. Combining the methods of this paper with Quillen's calculation of the cohomology of $\mr{GL}_n(\bF_q)$ away from the defining characteristic, we are also able to calculate the groups $H_*(\mr{GL}_n(\bF_q) ; \mr{St}(\bF_q^n) \otimes \bF_\ell)$ for $\ell \neq p$. We give this calculation in \cref{sec:homology-steinberg}.

\subsection*{Acknowledgments}

AK and SG were supported by the European Research Council (ERC) under the European Union's Horizon 2020 research and innovation programme (grant agreement No.\ 682922).  SG was also supported by NSF grant DMS-1405001, by the EliteForsk Prize, and by the Danish National Research Foundation through the Copenhagen Centre for Geometry and Topology (DNRF151). AK was also supported by the Danish National Research Foundation through the Centre for Symmetry and Deformation (DNRF92) and by NSF grant DMS-1803766. ORW was partially supported by EPSRC grant EP/M027783/1, and partially supported by the ERC under the European Union's Horizon 2020 research and innovation programme (grant agreement No.\ 756444), and by a Philip Leverhulme Prize from the Leverhulme Trust.

\tableofcontents

\section{Recollections on a homology theory for $E_\infty$-algebras}\label{sec:recollections}

In this section we informally explain that part of the theory developed in \cite{e2cellsIv3} which is necessary to prove Theorems \ref{thm:main-fp} and \ref{thm:main-f2}. We refer to that paper for a more formal discussion as well as for proofs; we shall refer to things labelled $X$ in \cite{e2cellsIv3} as $E_k$.$X$ throughout this paper. The reader may find it helpful to consult Sections 3 and 5 of \cite{e2cellsII}, which contain applications of this theory to mapping class groups that are similar in technique to the applications in this paper.

The aforementioned theory concerns CW approximation for $E_\infty$-algebras. Because we are interested mainly in the \emph{homology} of certain $E_\infty$-algebras, we will work in the category of \emph{simplicial $\bk$-modules}. In this paper $\bk$ shall always be a field, and we will also make this simplifying assumption in this section. In the case that $\bk = \bQ$, an $E_\infty$-algebra in simplicial $\bk$-modules is equivalent to the data of a commutative algebra in non-negatively graded chain complexes over $\bQ$. 

In order to eventually keep track of the dimension $n$ of the vector space $\bF^n$ we shall work with an additional $\bN$-grading (for us $\bN$ always includes $0$) which we call \emph{rank}, and so work in the category $\cat{sMod}^\bN_\bk$ of $\bN$-graded simplicial $\bk$-modules.  Morphisms in this category are weak equivalences if they are weak equivalences of underlying $\bN$-graded simplicial sets, and the category is equipped with the symmetric monoidal structure given by Day convolution. That is, an object consists of a sequence of simplicial $\bk$-modules $M_\bullet(n)$ for $n \in \bN$, and the tensor product of two such objects is given by
\[(M \otimes N)_p(n) = \bigoplus_{a+b=n} M_p(a) \otimes_\bk M_p(b).\]

A \emph{non-unital $E_\infty$-operad} in topological spaces is an operad whose space of $n$-ary operations is a contractible space with free $\fS_n$-action, with the exception of the space of $0$-ary operations, which is empty. An example of such an operad is provided by taking the colimit as $k \to \infty$ of the non-unital little $k$-cubes operads. Taking $\bk[\mr{Sing}(-)]$ we obtain an operad in simplicial $\bk$-modules, which we make $\bN$-graded by concentrating it in degree $0$. We denote the resulting operad by $\cC_\infty$. An \emph{$E_\infty$-algebra} in this paper shall usually mean an algebra over this operad in the category $\cat{sMod}^\bN_\bk$.

The module of \emph{indecomposables} of an $E_\infty$-algebra $\gR$ in $\cat{sMod}^\bN_\bk$ is defined using the exact sequence of $\bN$-graded simplicial $\bk$-modules
\[\bigoplus_{n \geq 2} \cC_\infty(n) \otimes \gR^{\otimes n} \lra \gR \lra Q^{E_\infty}(\gR) \lra 0\]
with the left map induced by the structure maps of $\gR$ as an $E_\infty$-algebra. This construction is not homotopy invariant, and has a derived functor $Q^{E_\infty}_\bL$. (In the notation of \cite{e2cellsIv3} we have actually defined the relative indecomposables, but because $\cC_\infty(1) \simeq *$ at the level of derived functors there is no difference between this and the absolute indecomposables, cf.\ equation (11.3) in Section $E_k$.11.3.) The \emph{$E_\infty$-homology groups} are then defined as $H^{E_\infty}_{n,d}(\gR) \coloneqq H_d(Q^{E_\infty}_\bL(\gR)(n))$; in simplicial $\bk$-modules, homology means taking the homology of the associated chain complex, or equivalently taking homotopy groups.

Its role is explained by Theorem $E_k$.11.21, which says that $E_\infty$-homology determines the cells needed to build a CW approximation to $\gR$ in $\cat{Alg}_{E_\infty}(\cat{sMod}_\bk^\bN)$. Let us explain what this means. We write $\partial D^{n,d} \in \cat{sMod}_\bk^\bN$ for the object which is 0 when evaluated at any $k \neq n$, and is the simplicial $\bk$-module $\bk[\partial \Delta^d]$ when evaluated at $n$; we write $D^{n,d}$ for the analogous construction using $\bk[\Delta^d]$, so there is an inclusion $\partial D^{n,d} \subset D^{n,d}$. A \emph{cellular $E_\infty$-algebra} is one obtained by iterated cell attachments, the data for which is a map of simplicial $\bk$-modules $\bk[\partial \Delta^d] \to \gR(n)$, which by adjunction gives a morphism $\partial D^{n,d}_\bk \to \gR$ in $\cat{sMod}_\bk^\bN$. Letting $\gE_\infty(-)$ denote the free $E_\infty$-algebra construction, this in turn gives a map $\gE_\infty(\partial D^{n,d}_\bk) \to \gR$ of $E_\infty$-algebras. The inclusion $\partial D^{n,d} \subset D^{n,d}$ induces the left map in the following diagram, and to attach a cell we form the pushout in $\cat{Alg}_{E_\infty}(\cat{sMod}^\bN_\bk)$
\[\begin{tikzcd}\gE_\infty(\partial D^{n,d}_\bk) \rar \dar & \gR \dar \\
\gE_\infty(D^{n,d}_\bk) \rar & \gR \cup^{E_\infty} \gD^{g,d}.\end{tikzcd}\]
A \emph{cellular approximation} to $\gR$ is a weak equivalence $\gC \overset{\sim}\to \gR$ from a cellular $E_\infty$-algebra. If $\gR(0) \simeq 0$ and $H_{n,d}^{E_\infty}(\gR)=0$ for $d < f(n)$, then it follows that $\gR$ admits a cellular approximation built only using cells of bidegrees $(n,d)$ such that $d \geq f(n)$.  This type of statement shall be called a \emph{vanishing line} for $E_\infty$ homology. In fact rather than just cellular, $\gC$ can be taken to be CW, meaning that it comes with a filtration where the $d$th stage is obtained from the $(d-1)$st by attaching $d$-cells. Such an approximation can be used for homology computations by taking the spectral sequence associated to this skeletal filtration, or using iterated cell attachment spectral sequences as in Section $E_k$.10.3.

\vspace{.5em}

To prove Theorems \ref{thm:main-fp} and \ref{thm:main-f2}, we shall apply these ideas to a certain $E_\infty$-algebra $\gR_{\bF_p}$ with $\gR_{\bF_p}(n) \simeq \bF_p[B\mr{GL}_n(\bF_q)]$. This is a simplicial $\bF_p$-module whose homology groups are the $\bF_p$-homology of $\mr{GL}_n(\bF_q)$.

We first consider $\gR_{\bF_p}$ as an
$E_1$-algebra, and prove a vanishing result for the $E_1$-homology
groups $H^{E_1}_{n,d}(\gR_{\bF_p})$, defined analogously to the
$E_\infty$-homology groups.  This is done by first studying
$\mr{GL}_n(\bF_q)$-equivariant semi-simplicial sets
$S^{E_1}(\bF^n_q)$, first introduced by Charney, who called them
``split buildings'' and who proved they had the homotopy type of a
wedge of $(n-2)$-spheres.  Then the homology groups of $Q^{E_1}_\bL(\gR_{\bF_p})$ may be calculated as those of the homotopy
orbit space $S^{E_1}(\bF^n_q)\moddd\mr{GL}_n(\bF_q)$, and the
vanishing line for $E_1$-homology is established by proving the
vanishing of $H_*(\mr{GL}_n(\bF_q);H_{n-2}(S^{E_1}(\bF^n_q);\bF_p))$
in a range of degrees. The resulting vanishing line for $E_1$-homology implies the same for $E_\infty$-homology, using bar spectral sequences as in Theorem $E_k$.14.1.

By combining such vanishing results with calculations in low degrees and low rank, we can produce an $E_\infty$-algebra with the same homological stability behavior as $\gR_{\bF_p}$ in a range. For $q = p^r \neq 2$ this $E_\infty$-algebra is the commutative algebra $\bN$ and $\gR_{\bF_p}$ maps to it. For $q = 2$  this is a custom-built $E_\infty$-algebra $\gA$ which maps to $\gR_{\bF_2}$.

\subsection{Glossary and some notation from \cite{e2cellsIv3}}
\label{sec:notat-other-recoll}

For the reader's convenience we collect some notation from op.\ cit., but we refer there for more details. This is borrowed from \cite{e2cellsIV}.

\begin{enumerate}[$\bullet$]
	\item Most of our arguments take place in the categories $\cat{sSet}^\bN$, $\cat{sSet}^\bN_*$, and $\cat{sMod}_\bk^\bN$, of functors from $\bN$ (regarded as a category with only identity morphisms) to the category of simplicial sets, pointed simplicial sets, and simplicial $\bk$-modules.
	\item Given an algebra $\gR \in \Alg_{E_\infty}(\cat{sSet}^\bN)$ and a commutative ring $\bk$, we denote by $\gR_\bk \in \Alg_{E_\infty}(\cat{sMod}_\bk^\bN)$ the $\bk$-linearisation of $\gR$.
	\item Objects $X \in \cat{sMod}_\bk^\bN$ have homology groups which are bigraded $\bk$-modules, defined by $H_{n,d}(X) = \pi_d(X(n))$.  
	\item The unitalisation of an $E_k$-algebra $\gR$ is denoted $\gR^+$, see Section $E_k$.4.4. By ``$E_k$-algebra'' we always mean the non-unital type, unless otherwise specified.  
	\item If $\gR$ is an $E_1$-algebra, then $\overline{\gR}$ denotes a unital and associative algebra naturally weakly equivalent to $\gR^+$ as a unital $E_1$-algebra, see Proposition $E_k$.12.9 and the construction preceding it. If $\gR$ is an $E_k$-algebra, we regard it is an $E_1$-algebra before applying these constructions.
	\item Filtered objects in a category $\cat{C}$ are functors $\bZ_{\leq} \to \cat{C}$, where $\bZ_{\leq}$ denotes the category with object set $\bZ$ and morphism set $m \to n$ either a singleton or empty, depending on whether $m \leq n$ or $m > n$.  We often consider filtered objects in functor categories such as $\cat{C} = \cat{sMod}_\bk^\bN$, in which case filtered objects are functors $\bN \times \bZ_{\leq} \to \cat{sMod}_\bk$.
	\item To an unfiltered object $X$ is associated a filtered object $a_* X$ for each $a \in \bZ$, making the functor $X \mapsto a_* X$ left adjoint to evaluation at the object $a \in \bZ_{\leq}$.  Explicitly, $(a_* X)(n)$ is $X$ for $n \geq a$ and is the initial object for $n < a$.  See Section $E_k.5.3.2$.
	\item There is a spectral sequence associated to a filtered object $X \in \cat{sMod}_\bk^{\bN \times \bZ_{\leq}}$.  In our grading conventions, see Theorem $E_k$.10.10, it has
	\begin{equation*}
		E^1_{n,p,q} = H_{n,p+q}(X(p),X(p-1)), \quad
		d^r \colon  E^r_{n,p,q} \to E^r_{n,p-r,q+r-1},
	\end{equation*}
	and if it conditionally converges, it does so to $H_{n,p+q}(\mr{colim}\,X)$. This does not relate the different $n$ at all (disregarding any multiplicative structures) so may be regarded as one spectral sequence for each $n \in \bN$. 
\end{enumerate}

\section{General linear groups as an $E_\infty$-algebra} \label{sec:einfty-algebra}

\begin{notation}In the paper we shall use $\mr{GL}(V)$ to denote the linear isomorphisms of a vector space $V$, and shall when appropriate write $\mr{GL}(\bF^n)$ for $\mr{GL}_n(\bF)$.\end{notation}

\subsection{Construction of $\gR$} 
Fix a field $\bF$, which need not be finite yet. We define a groupoid $\cat{V}_\bF$ with objects the non-negative integers, and morphisms given by
\[\cat{V}_\bF(n,m) = \begin{cases} \mr{GL}_n(\bF) & \text{if $n=m$,} \\
\varnothing & \text{otherwise.}\end{cases}\]

This is a skeleton of the category $\cat{P}_\bF$ with objects finitely dimensional $\bF$-vector spaces and morphisms linear isomorphisms, which has a symmetric monoidal structure given by direct sum. Let us describe the corresponding symmetric monoidal structure on this skeleton. On objects the monoidal product is given by $n \oplus m = n+m$ and on morphisms by block sum of matrices, as a homomorphism $\mr{GL}_n(\bF) \times \mr{GL}_m(\bF) \to \mr{GL}_{n+m}(\bF)$. This admits a symmetry given by the identity on objects, and on morphisms by conjugation with the block $(n+m) \times(n+m)$-matrix
\[T_{n,m} \coloneqq \begin{bmatrix} 0 & \mr{id}_n \\
\mr{id}_m & 0 \end{bmatrix},\]
that is, $A \mapsto T_{n,m}^{-1} A T_{n,m}$. This category has several useful properties.

\begin{lemma}\label{lem:vect-cat-props} The symmetric monoidal groupoid $(\cat{V}_\bF, \oplus, 0)$ has the following properties:
	\begin{enumerate}[(i)]
		\item for the monoidal unit $0$ we have $\mr{GL}_0(\bF) = \{e\}$, 
		\item the homomorphism $ -\oplus - \colon \mr{GL}_n(\bF) \times \mr{GL}_m(\bF) \to \mr{GL}_{n+m}(\bF)$ induced by the monoidal structure is injective,
		\item there is a single object in each isomorphism class, 
		\item the monoidal structure is strictly associative.
	\end{enumerate}
\end{lemma}

Consider the category $\cat{sSet}^{\cat{V}_\bF}$ of functors $\cat{V}_\bF \to \cat{sSet}$. This is simplicially enriched and inherits a symmetric monoidal structure via Day convolution. As consequence, for any simplicial operad $\cO$ there is a category $\cat{Alg}_\cO(\cat{sSet}^{\cat{V}_\bF})$ of $\cO$-algebras in $\cat{sSet}^{\cat{V}_\bF}$. Furthermore, $\cat{sSet}^{\cat{V}_\bF}$ admits a projective model structure, making it into a simplicially enriched symmetric monoidal model category. If $\cO$ is $\Sigma$-cofibrant, we may transfer this to obtain a projective model structure on $\cat{Alg}_\cO(\cat{sSet}^{\cat{V}_\bF})$ with the property that the forgetful functor $\cat{Alg}_\cO(\cat{sSet}^{\cat{V}_\bF}) \to \cat{sSet}^{\cat{V}_\bF}$ preserves cofibrations in addition to fibrations and weak equivalences.

In particular, we may take $\cO = \cC_\infty \coloneqq \mr{colim}_{k \to \infty} \, \cC_k$, where $\cC_k$ denotes the non-unital little $k$-cubes operad. There is a non-unital commutative monoid $\underline{\ast}_{>0}$ in $\cat{sSet}^{\cat{V}_\bF}$ given by
\[\underline{\ast}_{>0}(n) = \begin{cases}
\varnothing &\text{if $n=0$,}\\
* & \text{if $n > 0$.}
\end{cases}\]
This is in particular a non-unital $E_\infty$-algebra. As in Section $E_k$.17.1, by cofibrantly replacing $\underline{\ast}_{>0} \in\Alg_{E_\infty}(\cat{sSet}^\cat{\cat{V}_\bF})$ we construct a cofibrant non-unital $E_\infty$-algebra $\gT$ such that\[\gT(n) \simeq \begin{cases}
\varnothing & \text{ if $n=0$,}\\
\ast & \text{ if $n>0$.}
\end{cases}\]
Associating to an $\bF$-vector space its dimension gives a symmetric monoidal functor $r \colon \cat{V}_\bF \to \bN$. Left Kan extension provides a left adjoint $r_* \colon \cat{sSet}^{\cat{V}_\bF} \to \cat{sSet}^{\bN}$ to the restriction functor $r^* \colon \cat{sSet}^{\bN} \to \cat{sSet}^{\cat{V}_\bF}$, and this sends $E_\infty$-algebras to $E_\infty$-algebras so may also be considered as functor  $r_* \colon \cat{Alg}_{E_\infty}(\cat{sSet}^{\cat{V}_\bF}) \to \cat{Alg}_{E_\infty}(\cat{sSet}^{\bN})$. As such it is the left adjoint in a Quillen adjunction, and hence $r_*(\gT) \simeq \bL r_*(\underline{\ast}_{>0})$. We shall denote this $E_\infty$-algebra by $\gR$. Since the forgetful functor preserves cofibrant objects, we get
\[\gR(n) \simeq \begin{cases} \varnothing & \text{if $n=0$,} \\
B\mr{GL}_n(\bF) & \text{if $n>0$.}\end{cases}\]
This is the promised $E_\infty$-algebra structure on an $\bN$-graded simplicial set weakly equivalent to $\bigsqcup_{n \geq 1} B\mr{GL}(\bF^n)$. Compare to the case of Dedekind domains considered in Section $E_k$.18.2.

For any commutative ring $\bk$, we have that $H_*(\gR(n);\bk) \cong H_*(\mr{GL}_n(\bF);\bk)$ for $n >0$. Since we are interested only in these homology groups, there is no loss in passing from simplicial sets to simplicial $\bk$-modules and focusing our attention on $\gR_\bk \coloneqq \bk[\gR] \in \cat{Alg}_{E_\infty}(\cat{sMod}_\bk^\bN)$.

\subsection{Tits buildings and relative Tits buildings}

Though $\gR_\bk$ is an $E_\infty$-algebra, the map $\cC_1 \to \cC_\infty$ allows us to consider it as an $E_1$-algebra instead. As such it has $E_1$-homology, which may be computed in terms of the $E_1$-splitting complex of $\gR_\bk$ as in Section $E_k$.17.2. We will study these $E_1$-splitting complexes using Tits buildings and their relative versions, whose definitions we recall in this section.

\begin{definition}Let $P$ be a finite-dimensional $\bF$-vector space. The \emph{Tits building $\cT(P)$} is the poset with elements given by subspaces $0 \subsetneq V \subsetneq P$ and partial order given by $V \leq V'$ if $V \subseteq V'$.

We let $T(P)$ denote the topological space obtained by taking the thin geometric realisation of the nerve of $\cT(P)$.\end{definition}

See \cite[Section 1.2]{EbertRWSx} for background on (thin or thick) geometric realisations. The Solomon--Tits theorem says that $T(P)$ is homotopy equivalent to a wedge of $(\dim(P)-2)$-spheres \cite{solomon}. Again the only possibly non-vanishing reduced homology group is $\widetilde{H}_{\dim(P)-2}$, a $\bZ[\mr{GL}(P)]$-module that is free as an abelian group.

\begin{definition}The \emph{Steinberg module of $P$} is the $\bZ[\mr{GL}(P)]$-module 
\[\mr{St}(P) \coloneqq \widetilde{H}_{\dim P-2}(T(P)).\]
\end{definition}

In order to make certain inductive arguments later, we will also need relative versions of the Tits building. These are subposets of the Tits building consisting of subspaces complementary to a fixed subspace $W$.

\begin{definition}\label{def:rel-tits-building} Let $P$ be a finite-dimensional $\bF$-vector space and $W \subset P$ be a subspace. Then the \emph{relative Tits building $\titsrel{P}{W}$} is the poset with elements given by non-zero proper subspaces $V$ of $P$ such that $V \cap W = \{0\}$ and partial order given by $V \leq V'$ if $V \subseteq V'$.
	
We let $\trel{P}{W}$ denote the topological space obtained by taking the thin geometric realisation of the nerve of $\titsrel{P}{W}$.	
\end{definition}

In \cref{lem:rel-tits-connectivity}, we will prove that $\trel{P}{W}$ is homotopy equivalent to a wedge of $(\dim P-\dim W-1)$-spheres when $W$ is a non-trivial proper subspace. It admits an action of the subgroup 
\[\mr{GL}(P, \text{pres $W$}) \leq\mr{GL}(P)\]
consisting of linear isomorphisms of $P$ which preserve $W$ as a subspace (but do not necessarily fix it pointwise).

\begin{definition}
The \emph{Steinberg module of $P$ relative to $W$} is the $\bZ[\mr{GL}(P,\text{pres $W$})]$-module 
\[\strel{P}{W} \coloneqq \widetilde{H}_{\dim P-\dim W-1}(\trel{P}{W}).\]
\end{definition}

The following is the ``linear dual'' of the poset $\titsrel{P}{W}$ as in \cref{def:rel-tits-building}.

\begin{definition}\label{def:rel-tits-dual} Let $P$ be a finite-dimensional $\bF$-vector space and $W \subset P$ be a subspace. Then the \emph{dual relative Tits building $\titsdualrel{P}{W}$} is the poset with elements given by non-zero proper subspaces $V$ of $P$ such that $V+W = P$ and partial order given by $V \leq V'$ if $V \subseteq V'$.

We let $\tdualrel{P}{W}$ denote the topological space obtained by taking the thin geometric realisation of the nerve of $\titsdualrel{P}{W}$.\end{definition}

If $W = P$ this is equal to the Tits building $\cT(P)$ and has a $\mr{GL}(P)$-action. For proper non-trivial $W$, it has an action of the subgroup $\mr{GL}(P,\text{ pres $W$})$ of $\mr{GL}(P)$ consisting of linear isomorphisms of $P$ which preserve $W$.

By dualizing we can identify it with a relative Tits building. Letting $W^\circ \subset V^\vee$ denote the linear subspace of linear forms which are identically $0$ on $W$, there are mutually inverse functors
\begin{align*}\titsrel{P}{W} &\overset{\cong}\longleftrightarrow \titsdualrel{P^\vee}{W^\circ} \\
V & \longmapsto V^\vee \\
U^\vee & \longmapsfrom U,\end{align*}
where the bottom map uses the canonical identification of the double dual $(P^{\vee})^\vee$ with $P$. These identifications are equivariant for the action of $\mr{GL}(P,\text{ pres $W$})$ on the left hand side and the action of $\mr{GL}(P^\vee,\text{ pres $W^\circ$})$ on the right side, upon identifying these groups by dualizing.

One can reduce the study of relative Tits buildings to the study of ordinary Tits buildings and those relative Tits buildings where $W \subset P$ is a line. This is due to Quillen and appears in his unpublished notebooks. We give a version of his proof below, and a similar argument can be found in \cite[\S 3.1]{SprehnWahl}:

\begin{lemma}[Quillen] \label{lem:p-decomposition} Let $P$ be a $\bF$-vector space of finite dimension $\geq 2$ and $0 \subsetneq W \subsetneq P$ be a non-zero proper subspace. Let $L$ be a one-dimensional subspace of $W$. Then there is a homotopy equivalence
	\begin{equation} \label{eqn:quillen-p-decomposition} \tdualrel{P}{W} \simeq \tdualrel{P/L}{W/L} \ast \tdualrel{P}{L},\end{equation}
	equivariant for the action of the subgroup $\mr{GL}(P,\text{ pres $L$, pres $W$}) \leq \mr{GL}(P,\text{ pres $W$})$ of those linear maps which also preserve the subspace $L$.
\end{lemma}

\begin{proof}
To prove this homotopy equivalence, we define a pair of subposets: (i) $\cL_0 \subset \titsdualrel{P}{W}$ is the full subposet consisting of those $V$ such that $L \subset V$, (ii) $\cL_1 \subset \titsdualrel{P}{W}$ is the full subposet consisting of those $V$ such that $V+L \neq P$. We write $L_i$ for the thin geometric realisation of the nerve of $\cL_i$.
	
The poset $\cL_0$ is a subposet of $\cL_1$ and there is deformation retraction $L_1 \to L_0$ induced by the relation $V \leq V + L$, regarded as a natural transformation from the identity functor on the poset $\cL_0$ regarded as a category, to the functor $V \mapsto V + L$.  Hence
% so that
$L_0 \simeq L_1$.
Similarly, there are mutually inverse functors
	\begin{align*}\cL_0 &\overset{\cong}{\longleftrightarrow} \titsdualrel{P/L}{W/L} \\
	V&\longmapsto V/L \\
	p^{-1}(V/L) &\longmapsfrom V/L,\end{align*}
	where $p \colon P \to P/L$ is the projection map. These exhibit $\cL_0$ and $\titsdualrel{P/L}{W/L}$ as posets which are isomorphic, equivariantly for the action of $\mr{GL}(P,\text{ pres $L$, pres $W$})$.
	
	The poset $\cL_1$ is a full subposet of $\tdualrel{P}{W}$ and the only missing elements are the rank $(m-1)$ subspaces $H$ satisfying $H+L = P$, i.e.\ the complements to $L$. If we define subposets $\mr{Link}_{\cL_1}(H) \subset \cL_1$ to be given by those $0 \subsetneq V \subsetneq H$ such that $V+L \neq P$, then $\tdualrel{P}{W}$ is obtained from $L_1$ by attaching cones on the geometric realisations $|\mr{Link}_{\cL_1}(H)|$ of their nerves. That is, there is a pushout diagram
	\[\begin{tikzcd}\underset{H\, \mid H+L=P}\bigsqcup \dar |\mr{Link}_{\cL_1}(H)| \rar \dar & L_1 \dar \\
	\underset{H\, \mid H+L=P}\bigsqcup \mr{Cone}(|\mr{Link}_{\cL_1}(H)|) \rar & \tdualrel{P}{W}.\end{tikzcd}\] 
	
	We claim that the inclusion $\mr{Link}_{\cL_1}(H) \hookrightarrow \cL_1$ is a homotopy equivalence. Its homotopy inverse is given by $V \mapsto (V+L) \cap H$. To see this is well-defined, we need to check that (a) $(V + L) \cap H$ can be neither $0$ nor $H$, and (b) $(V+L) \cap H + L \neq P$. For (a), if $(V+L) \cap H = 0$, then $V+L$ must have rank 1 and hence $V = L$, but then $V+W = P$ can not hold, as $W \subsetneq P$. Similarly, if $(V + L) \cap H = H$, then $L \subset V + L \subset H$ and $H+L \neq P$. For (b), if $(V + L) \cap H+L$ has dimension $\dim(P)$, then $(V+L) \cap H$ has dimension $\dim(P)-1$, hence must equal $H$ and we saw above this can not happen. The homotopy is given by the zigzag $V \leq V + L \geq (V + L) \cap H$.
	
	Because the attaching maps are homotopy equivalences, the pushout $\tdualrel{P}{W}$ is homotopy equivalent to the join of $L_1$ with the set $\{H \mid H+L = P\}$. The first of these is homotopy equivalent to $\tdualrel{P/L}{W/L}$ and the second is isomorphic to $\tdualrel{P}{L}$, so their join is homotopy equivalent to $\tdualrel{P/L}{W/L} \ast \tdualrel{P}{L}$ as required.\end{proof}

By induction on $\dim W$ using the fact that $\titsdualrel{P}{L}$ is non-empty as long as $\dim(P) \geq 2$, we obtain the following result, which was also obtained by Vogtmann \cite[Proposition 1.1]{vogtmannspherical}:

\begin{lemma}\label{lem:rel-tits-connectivity} Let $P$ be a $\bF$-vector space of finite dimension $\geq 2$ and $0 \subsetneq W \subsetneq P$ be a non-zero proper subspace, then $\tdualrel{P}{W}$ is $(\dim W-1)$-spherical.\end{lemma}

\subsection{$E_1$-splitting complexes} Using Remark $E_k$.17.11 %$E_k$.17.2.10 
and the fact that the symmetric monoidal groupoid $\cat{V}_\bF$ satisfies the conditions in \cref{lem:vect-cat-props}, the $E_1$-splitting complex of $\gR_\bk$ can be described in terms of Young-type subgroups. The spaces $S^{E_1}(P)$ obtained this way were originally invented by Charney in analogy with the Tits building; in the Tits building one considers proper subspaces rather than proper splittings of a vector space.

\begin{definition}\label{def:e1splitting}
	Let $P$ be a finite-dimensional $\bF$-vector space. The \emph{$E_1$-splitting complex $S^{E_1}_\bullet(P)$} is the semisimplicial set with $p$-simplices given by ordered collections $(V_0,\ldots,V_{p+1})$ of non-zero proper subspaces of $P$, such that the natural map $V_0 \oplus \ldots \oplus V_{p+1} \to P$ induced by the inclusions is an isomorphism. The face maps $d_i$ takes the sum of the $i$th and $(i+1)$st term.
	
	We let $S^{E_1}(P)$ denote the topological space obtained by taking the thick geometric realisation of $S^{E_1}_\bullet(P)$.
\end{definition}

Charney proved that $S^{E_1}(P)$ is a wedge of $(\dim(P)-2)$-dimensional spheres \cite[Theorem 1.1]{Charney}. Thus the only possibly non-vanishing reduced homology group is $\widetilde{H}_{\dim(P)-2}$, which is a $\bZ[\mr{GL}(P)]$-module that is free as an abelian group.

\begin{definition}
	The \emph{$E_1$-Steinberg module of $P$} is the $\bZ[\mr{GL}(P)]$-module 
	\[\mr{St}^{E_1}(P) \coloneqq \widetilde{H}_{\dim(P)-2}(S^{E_1}(P)).\]
\end{definition}

As explained in \cref{sec:recollections}, the $E_1$-homology groups $H^{E_1}_{n,d}(\gR_\bk)$ are the homology groups of the derived $E_1$-indecomposables of $\gR_\bk$, and by Corollary $E_k$.17.5, Lemma $E_k$.17.10, and Remark $E_k$.17.11, these may be described in terms of the homotopy orbit space $S^{E_1}(\bF^n) \moddd \mr{GL}_n(\bF)$. Charney's connectivity result for $S^{E_1}(\bF^n)$ then implies that the homology groups of $\mr{GL}_n(\bF)$ with coefficients in $\mr{St}^{E_1}(\bF^n)$ coincide with $E_1$-homology groups of $\gR_\bk$ up to a shift, c.f.~the formula following Definition $E_k$.17.6:
\begin{equation}\label{eqn:e1homology-ste1}
	H^{E_1}_{n,d}(\gR_\bk) \cong H_{d-n+1}(\mr{GL}_n(\bF);\mr{St}^{E_1}(\bF^n) \otimes \bk).
\end{equation}
It is useful to think of the semi-simplicial set $S^{E_1}_\bullet(P)$ in terms of the following poset.

\begin{definition}Let $P$ be a finite-dimensional $\bF$-vector space. The \emph{split building $\cS^{E_1}(P)$} is the poset with objects given by ordered pairs $(V_0,V_1)$ of non-zero proper subspaces of $P$ such that $V_0 \cap V_1 = \{0\}$ and $V_0 \oplus V_1 \to P$ is an isomorphism. The partial order is given by $(V_0,V_1) \leq (V'_0,V'_1)$ if $V_0$ is a subspace of $V'_0$ and $V'_1$ is a subspace of $V_1$.
\end{definition}

We recognise the semi-simplicial set $S^{E_1}_\bullet(P)$ as the non-degenerate simplices of the nerve $N_\bullet \cS^{E_1}(P)$, so the thick geometric realisation of $S^{E_1}_\bullet(P)$ is homeomorphic to the thin geometric realisation of $N_\bullet \cS^{E_1}(P)$.

\section{Computing $E_1$-homology of general linear groups} \label{sec:e1-splitting-generic}

In this section we describe a method to prove that the $E_1$-splitting complex of a finite-dimensional vector space is highly-connected and that homology with coefficients in the $E_1$-Steinberg module vanishes in a range. The connectivity result is due to Charney \cite{Charney}, who proves it more generally for Dedekind domains. In fact, we follow her proof strategy but additionally keep track of various homology groups with coefficients.

\subsection{Poset techniques}\label{sec:posets} Recall that the $E_1$-splitting complex $S^{E_1}(\bF^n)$ of \cref{def:e1splitting} can be described as the geometric realisation of the nerve of the poset $\cS^{E_1}(\bF^n)$. There are well-developed techniques to prove the connectivity of posets and maps between them (that is, upon taking geometric realisation of the nerves).

Recall that if $(\cY, \leq)$ is a poset and $y \in \cY$ then $\cY_{< y} \coloneqq \{y' \in \cY \setminus \{y\} \mid y' \leq y\}$ with the induced ordering, and similarly for $\cY_{> y}$.  If $f \colon \cX \to \cY$ is a map of posets then $f_{\leq y} \coloneqq \{x \in \cX \mid f(x) \leq y\}$ and similarly for $f_{\geq y}$. We shall use the following slight generalization of \cite[Theorem 9.1]{quillenposet}, which can be obtained using the mapping cone device of \cite[\S 2]{LvdK}. This has the assumption that $\cY$ admits a bounded height function \cite[p.~201]{LvdK}; a sufficient condition for this is that $\cY_{\leq y}$ is finite-dimensional for each $y$. This will be case for all examples considered in this paper.

\begin{theorem}\label{thm:nerve-spherical} Let $f \colon \cX \to \cY$ be a map of posets, $n \in \bZ$, and $t_\cY \colon \cY \to \bZ$ be a function. Assume that 
\begin{enumerate}[(i)]
		\item $\cY$ is $n$-spherical and admits a bounded height function, 
		\item for every $y \in \cY$, $f_{\leq y}$ is $(n-t_\cY(y))$-spherical, and
		\item $\cY_{>y}$ is $(t_\cY(y)-1)$-spherical. \end{enumerate}
Then $\cX$ is $n$-spherical and there is a canonical filtration
	\[0 = F_{n+1} \subset F_n \subset \cdots \subset F_{-1} = \widetilde{H}_n(\cX)\]
	such that 
	\begin{align*}F_{-1}/F_0 &\cong \widetilde{H}_n(\cY), \\
	F_{q}/F_{q+1} &\cong \bigoplus_{t_\cY(y) = n-q} \widetilde{H}_{n-q-1}(\cY_{>y}) \otimes \widetilde{H}_{q}(f_{\leq y}).\end{align*}
\end{theorem}

Observe we do \emph{not} assume there is a bounded height function on $\cX$.

\subsection{The statement}

We find it clarifying to use the notation 
\[\cS^{E_1}(\sdash,\sdash|\bF^n) \coloneqq \cS^{E_1}(\bF^n)\]
for the split building, as this notation allows us to easily denote several variants.

\begin{definition}Let $W \subset \bF^n$ be a non-zero subspace. The \emph{relative split building} $\cS^{E_1}(\sdash,W \subset \sdash|\bF^n)$ is the subposet of $\cS^{E_1}(\bF^n)$ with objects those splittings $(V_0,V_1)$ such that $W \subset V_1$.

We let $S^{E_1}(\sdash,W \subset \sdash |\bF^n)$ denote the thin geometric realisation of its nerve.
\end{definition}

In the following theorem we prove that this is homotopy equivalent to a wedge of $(n-\dim W-1)$-spheres, and we write 
\[\steonerel{\bF^n}{W} \coloneqq \widetilde{H}_{n-\dim W-1}(S^{E_1}(\sdash,W \subset \sdash |\bF^n)).\]
This notation is in analogy with $\strel{\bF^n}{W}$, as the objects of $\cS^{E_1}(\sdash,W \subset \sdash |\bF^n)$ are pairs $(A,B)$ which in particular satisfy $A \cap W = \{0\}$. This has an action of the subgroup $\mr{GL}(P, \text{pres $W$}) \leq\mr{GL}(P)$ consisting of linear isomorphisms of $P$ which preserve $W$ as a subspace, and hence also of the subgroup 
\[\mr{GL}(P, \text{fix $W$}) \leq\mr{GL}(P, \text{pres $W$})\]
 consisting of linear isomorphisms of $P$ which fix $W$ pointwise.

\begin{theorem}\label{thm.steinbergcoinv}
Let $\bF$ be a field, let $W \subset \bF^n$ be a non-zero subspace and suppose $n \geq 1$. Then we have that:
\begin{align*}S^{E_1}(\sdash,\sdash|\bF^n) & \text{ is $(n-2)$-spherical,} \\
S^{E_1}(\sdash,W \subset \sdash|\bF^n) & \text{ is $(n-\dim W-1)$-spherical.}\end{align*}

Suppose further that $\bk$ and $c \in \bN$ are such that
\begin{align*}
H_i(\mr{GL}(\bF^n);\mr{St}(\bF^n) \otimes \bk) &= 0,\\
H_i(\mr{GL}(\bF^n,\text{ fix $W$});\strel{\bF^n}{W} \otimes \bk) &=0,
\end{align*}
for all $n \geq 2$ and all $i \leq c$. Then for all $n \geq 2$ and $i \leq c$ it is true that:
\begin{align*}H_i(\mr{GL}(\bF^n);\mr{St}^{E_1}(\bF^n) \otimes \bk) &=0, \\
H_i(\mr{GL}(\bF^n,\text{ fix $W$});\steonerel{\bF^n}{W}) \otimes \bk) &=0.\end{align*}\end{theorem}

\begin{remark}A Serre spectral argument shows $H_i(\mr{GL}(\bF^n,\text{ fix $W$});\steonerel{\bF^n}{W} \otimes \bk)=0$ for $i \leq c$ implies $H_i(\mr{GL}(\bF^n,\text{ pres $W$});\steonerel{\bF^n}{W} \otimes \bk)=0$ for $i \leq c$. Hence the former is a stronger statement.\end{remark}
	
It shall be helpful to name two facts that are used in the proof of \cref{thm.steinbergcoinv}, as well as the two hypotheses in the second part of  \cref{thm.steinbergcoinv}:
\begin{enumerate}[\indent (1)]
	\item \label{enum.titsspherical} The Tits building $\cT(\bF^n)$ is $(n-2)$-spherical.
	\item \label{enum.reltitsspherical} The relative Tits building $\titsrel{\bF^n}{W}$ is $(n-\dim W-1)$-spherical.
\end{enumerate}

\begin{enumerate}[\indent (1)$^\mr{St}$]
	\item \label{enum.steinbergcoinv} For $n \geq 2$ and $i \leq c$, $H_i(\mr{GL}(\bF^n);\mr{St}(\bF^n) \otimes \bk) = 0$.
	\item \label{enum.relsteinbergcoinv} For $n \geq 2$ and $i \leq c$, $H_i(\mr{GL}(\bF^n,\text{ fix $W$});\strel{\bF^n}{W} \otimes \bk)$.
\end{enumerate}

\noindent Without loss of generality we can change basis of $\bF^n$ so that $W=\bF^w \times \{0\}$, the span of the first $w$ basis vectors. The proof of \cref{thm.steinbergcoinv} will be an upwards induction over $(n,w)$ in lexicographic order, where we interpret $w=0$ as the absolute version.

\subsubsection{The base cases} The base cases are those with $n \leq 1$, in which case all statements are empty.

\subsubsection{The first reduction} 

\begin{proposition}\label{prop:first-reduction-general} 
The case $(n,0)$ is implied by the cases $\{(n,w) \, | \, w >0\}$.\end{proposition}

\begin{proof}
Consider the map of posets
\begin{align*} \cX \coloneqq \cS^{E_1}(\sdash,\sdash|\bF^n) &\overset{f}\lra \cY \coloneqq \cT(\bF^n)^\mr{op} \\
(A,B) &\longmapsto B,
\end{align*}
and for $U \in \cY$ let $t_\cY(U) = \dim(U)-1$. We verify the assumptions of \cref{thm:nerve-spherical}: 
\begin{enumerate}[(i)]
	\item $\cY = \cT(\bF^n)^\mr{op}$ is $(n-2)$-spherical by (\ref{enum.titsspherical}).

	\item $f_{\leq U}$ is given by those $(A,B)$ such that $U \subset B$, i.e.\ $\cS^{E_1}(\sdash,U\subset \sdash|\bF^n)$. By the hypothesis this is $(n-\dim(U)-1) = (n-2-t_\cY(U))$-spherical. 

	\item $\cY_{>U}$ is given by $\cT(U)$ so is $(\dim(U)-2) = (t_\cY(U)-1)$-spherical by (\ref{enum.titsspherical}).
\end{enumerate}
\cref{thm:nerve-spherical} then implies that $\cX=\cS^{E_1}(\sdash,\sdash|\bF^n)$ is $(n-2)$-spherical, and that there is a filtration of the $\bk[\mr{GL}(\bF^n)]$-module $\mr{St}^{E_1}(\bF^n)$ with $F_{-1}/F_0 = \mr{St}(\bF^n) \otimes \bk$ and the other filtration quotients $F_q/F_{q+1}$ given by
\begin{align*}
&\bigoplus_{\substack{0 < U < \bF^n\\ \dim(U) = n-q-1}} \mr{St}(U) \otimes \widetilde{H}_{q}(S^{E_1}(\sdash,U \subset \sdash|\bF^n)) \otimes \bk\\
&\qquad\cong \mr{Ind}^{\mr{GL}(\bF^n)}_{\mr{GL}(\bF^n,\text{ pres $\bF^{n-q-1}$})} \left(\mr{St}(\bF^{n-q-1}) \otimes \steonerel{\bF^n}{\bF^{n-q-1}} \otimes \bk\right),
\end{align*}
where $\bF^{n-q-1} \subset \bF^n$ is spanned by the first $(n-q-1)$ basis vectors.
	
For any filtered $\bk[G]$-module $F_{-1} = M \supset F_0 \supset \cdots \supset F_n \supset 0$ there is a spectral sequence
\[E^1_{p,q} = H_{p+q}(G ; F_q/F_{q+1}) \Longrightarrow H_{p+q}(G; M),\]
which we apply to the above filtration of the $\bk[\mr{GL}(\bF^n)]$-module $\mr{St}^{E_1}(\bF^n)$. We can distinguish two different types of columns on the $E^1$-page. Firstly, for $q=-1$ we have $E^1_{p, -1} = H_{p-1}(\mr{GL}_m(\bF);\mr{St}(\bF^m) \otimes \bk)$ which vanishes for $p-1 \leq c$  and $n \geq 2$ by (\ref{enum.steinbergcoinv})$^\mr{St}$. Secondly, for $q\geq 0$ we can use Shapiro's lemma to identify 
\[E^1_{p,q} \cong H_{p+q}\left(\mr{GL}(\bF^n,\text{ pres $\bF^{n-q-1}$});\mr{St}(\bF^{n-q-1}) \otimes \steonerel{\bF^n}{\bF^{n-q-1}} \otimes \bk\right).\]

To prove that these vanish for $p+q \leq c$ we consider the extension
\[\mr{GL}(\bF^n, \text{ fix $\bF^{n-q-1}$}) \lra \mr{GL}(\bF^n, \text{pres $\bF^{n-q-1}$}) \lra \mr{GL}(\bF^{n-q-1})\]
and apply the Serre spectral sequence
\[\begin{tikzcd}
F^2_{s,t} =  H_s\left(\mr{GL}(\bF^{n-q-1});\parbox{8cm}{\centering$\mr{St}(\bF^{n-q-1}) \otimes$ \\[-3pt]
	$H_t\left[\mr{GL}(\bF^n, \text{ fix $\bF^{n-q-1}$}); \steonerel{\bF^n}{\bF^{n-q-1}} \otimes \bk\right]$}\right) \arrow[Rightarrow]{d} \\ H_{s+t}\left(\mr{GL}(\bF^n,\text{ pres $\bF^{n-q-1}$}) ; \mr{St}(\bF^{n-q-1}) \otimes \steonerel{\bF^n}{\bF^{n-q-1}} \otimes \bk\right).
\end{tikzcd}\]
Now $F^2_{s,t}=0$ for $t \leq c$ by inductive assumption (\ref{enum.relsteinbergcoinv})$^\mr{St}$ as $n \geq 2$, so the target vanishes for $s+t \leq c$, and the claim follows.
\end{proof}

\subsubsection{Cutting down} The following technique is due to Charney \cite[p.\ 4]{Charney}. Given $V$ and $W$ subspaces of $P$, let $\cS^{E_1}(\sdash \subset V,W \subset \sdash|P)$ be the subposet of $\cS^{E_1}(\sdash,\sdash|P)$ consisting of those splittings $(A,B)$ such that $A \subset V$ and $W \subset B$.

\begin{lemma}\label{lem:cutting-down} Let $V$ and $W$ be subspaces of $P$ such that $V \cap W = 0$ and $V \oplus W \neq P$. Let $C$ be a complement to $W$ in $P$ containing $V$. Then there is an isomorphism
	\[\cS^{E_1}(\sdash \subset V,W \subset \sdash|P) \cong \cS^{E_1}(\sdash \subset V,\sdash|C),\]
equivariantly for the subgroup of $\mr{GL}(P)$ preserving the subspaces $V$, $W$ and $C$.\end{lemma}

\begin{proof}Mutually inverse functors are given by
	\begin{align*}\cS^{E_1}(\sdash \subset V,W \subset \sdash|P) & \overset{\cong}{\longleftrightarrow} \cS^{E_1}(\sdash \subset V,\sdash|C) \\
	(A,B) & \longmapsto (A,B \cap C) \\
	(A',B' \oplus W) & \longmapsfrom (A',B') \end{align*} 
This is well-defined, as it can not happen that $B \cap C = \{0\}$, because $V \oplus W \neq P$ implies that $B \cap C$ has dimension at least $1$.\end{proof}

\subsubsection{The second reduction} 

\begin{proposition}\label{prop:second-reduction-general} 
The case $(n,w)$ is implied by the cases $\{(n', w') \mid n' < n\}$.
\end{proposition}  

\begin{proof} 
When $w = n-1$, $\cS^{E_1}(\sdash,\bF^w \subset \sdash|\bF^n)$ is isomorphic to $\titsrel{\bF^n}{\bF^w}$ and we already know the result. So let us assume that $w \leq n-2$. 

Consider the map  of posets
\begin{align*}\cX \coloneqq \cS^{E_1}(\sdash,\bF^w \subset \sdash|\bF^n) &\overset{f}\lra \cY \coloneqq \titsrel{\bF^n}{\bF^w} \\
(A,B) &\longmapsto A, \end{align*}
%let $n = m-\dim W-1$ 
and take $t_\cY(V) = n-\dim V-w$. We verify the assumptions of \cref{thm:nerve-spherical}:
\begin{enumerate}[(i)]
%	\item $h(V) = \dim V-1$,
	\item $\titsrel{\bF^n}{\bF^w}$ is $(n-w-1)$-spherical by (\ref{enum.reltitsspherical}).

	\item $f_{\leq V}$ is $\cS^{E_1}(\sdash \subset V,\bF^w \subset \sdash|\bF^n)$, which we claim is $(\dim V-1)= (n-w-1-t_\cY(V))$-spherical.
	
	There are two cases: the first case is that $V \oplus \bF^w = \bF^n$, and then the complex is contractible (the splitting $(V,\bF^w)$ is terminal) and hence is a wedge of no $(n-w-1-t_\cY(V))$-spheres. The second case is that $V \oplus \bF^w \neq \bF^n$. Taking a complement $C$ to $\bF^w$, we can apply the cutting down lemma to see it is isomorphic to $\cS^{E_1}(\sdash \subset V,\sdash|C)$. Dualizing makes this isomorphic to $\cS^{E_1}(\sdash,V^\circ \subset \sdash|C^\vee)$, which is $(n-w-(w-\dim V)-1)=(\dim (V)-1)$-spherical by induction.

	\item $\cY_{>V}$ is isomorphic to $\titsrel{\bF^n/V}{\bF^w/V}$, which is $(n - \dim V - w - 1) = (t_\cY(V)-1)$-spherical by (\ref{enum.reltitsspherical}). Here and below we write $\bF^w/V$ for the image of $\bF^w$ under the projection map $\bF^n \to \bF^n/V$, and since $\bF^w \cap V = 0$ the vector space $\bF^w/V$ remains $w$-dimensional.
\end{enumerate} 
\cref{thm:nerve-spherical} then implies that $\cX = \cS^{E_1}(\sdash,\bF^w \subset \sdash|\bF^n)$ is $(n-w-1)$-spherical, and gives a filtration of the $\bk[\mr{GL}(\bF^n, \text{ fix }\bF^w)]$-module $\steonerel{\bF^n}{\bF^w}$ with $F_{-1}/F_0 = \strel{\bF^n}{\bF^w} \otimes \bk$ and further filtration quotients $F_{q-2}/F_{q-1}$ given by
\begin{align*}&\bigoplus_{\substack{0 < V < \bF^n \\ V \cap \bF^w = 0 \\ \dim V=q-1}} \strel{\bF^n/V}{\bF^w/V} \otimes \widetilde{H}_{q-2}(\cS^{E_1}(\sdash \subset V,\bF^w \subset \sdash|\bF^n)) \otimes \bk \\
&\qquad \cong \mr{Ind}^{\mr{GL}(\bF^n,\text{ fix $\bF^w$})}_{\mr{GL}(\bF^n,\text{ fix $\bF^w$, pres $V$})} \left(\shortstack{$\strel{\bF^n/V}{\bF^w/V} \otimes$ \\
	$\widetilde{H}_{q-2}(\cS^{E_1}(\sdash \subset V,\bF^w \subset \sdash|\bF^n)) \otimes \bk$}\right).\end{align*}
for $V$ a choice of $(q-1)$-dimensional subspace not intersecting $\bF^w$ (for example, the span of the last $q-1$ basis vectors).
	
The filtration quotient $F_{-1}/F_0 = \strel{\bF^m}{\bF^w} \otimes \bk$ has trivial $\mr{GL}(\bF^m,\text{ fix $\bF^w$})$-homology in degrees $i \leq c$ by (\ref{enum.relsteinbergcoinv})$^\mr{St}$. For the remaining filtration quotients, we may assume that $V \oplus \bF^w \neq \bF^n$ (otherwise the complex is contractible). Then we use Shapiro's lemma to rewrite $H_*(\mr{GL}(\bF^m,\text{ fix $\bF^w$}) ; F_{q-2}/F_{q-1})$ as the homology of the action of $\mr{GL}(\bF^m, \text{ fix $\bF^w$, pres $V$})$ on
	\[ \strel{\bF^n/V}{\bF^w/V} \otimes \widetilde{H}_{q-2}(\cS^{E_1}(\sdash \subset V,\bF^w \subset \sdash|\bF^n)) \otimes \bk.\] 
To show that this vanishes in a range of degrees, define a group $K$ by the extension
	\[0 \lra K \lra \mr{GL}(\bF^n, \text{fix $\bF^w$, pres $V$}) \lra \mr{GL}(\bF^n/V, \text{fix $\bF^w/V$}) \lra 0,\]
that is, $K=\mr{GL}(\bF^n, \text{fix $\bF^w$, pres $V$, fix $\bF^n/V$})$ (see \cref{fig:k-matrix} for an example). This groups acts trivially on $\strel{\bF^n/V}{\bF^w/V}$, so using the spectral sequence for this group extension it is enough to show vanishing for $i \leq c$ of
\[H_i\left(K ; \widetilde{H}_{\dim V-1}(\cS^{E_1}(\sdash \subset V,\bF^w \subset \sdash|\bF^n))\otimes \bk\right).\]
	
If $C$ is a complement to $\bF^w$ which contains $V$, then each element of $K$ has to preserve $C$ because $K$ acts as the identity on $\bF^n/V$. Thus  \cref{lem:cutting-down} gives a $\bk[K]$-module isomorphism
	\[\widetilde{H}_{\dim V-1}(\cS^{E_1}(\sdash \subset V,\bF^w \subset \sdash|\bF^n)) \otimes \bk \cong \widetilde{H}_{\dim V-1}(\cS^{E_1}(\sdash \subset V,\sdash|C)) \otimes \bk.\]
Furthermore the natural map $K \to \mr{GL}(C)$ is an isomorphism onto its image $\mr{GL}(C, \text{ pres $V$, fix $C/V$})$. Thus we are reduced to showing the vanishing of
	\[H_i\left(\mr{GL}(C, \text{pres $V$, fix $C/V$}) ; \widetilde{H}_{\dim V-1}(\cS^{E_1}(\sdash \subset V,\sdash|C)) \otimes \bk\right).\]
Dualising identifies this with
	\[H_i\left(\mr{GL}(C^\vee, \text{fix $V^\circ$}) ; \steonerel{C^\vee}{V^\circ} \otimes \bk\right)\]
which vanishes in the required range by induction as $\dim C < n$ because $C$ is a complement to the non-trivial $\bF^w$. To see that the hypothesis $\dim C = n-w \geq 2$ holds, recall that we assumed $w \leq n-2$ at the beginning of this proof.
\end{proof}

\begin{figure}
	\[\begin{bmatrix}
	1 & 0 & 0 & 0 & 0 \\
	0 & 1 & 0 & 0 & 0 \\
	0 & 0 & 1 & 0 & 0 \\
	0 & 0 & \ast & \ast & \ast \\
	0 & 0 & \ast & \ast & \ast \\
	\end{bmatrix}\]
	\caption{An element of $K$ as in the proof of \cref{prop:second-reduction-general} with $n=5$, $w=2$, and $V=\langle e_4, e_5\rangle$.}
	\label{fig:k-matrix}
\end{figure}

\section{General linear groups of finite fields except $\bF_2$} 
Throughout this section, we consider the field $\bF_{q}$ with $q = p^r \neq 2$.

\subsection{$E_1$-Steinberg homology}\label{sec:e1-steinberg-fp} 

Our first goal is to establish the following vanishing theorem for the homology of the $E_1$-Steinberg module.

\begin{theorem}\label{thm:gl-finite-coinvariants-simple}
$H_d(\mr{GL}_n(\bF_q);\mr{St}^{E_1}(\bF_q^n) \otimes \bF_p) =0$ for $d < r(p-1)-1$ and $n \geq 2$.
\end{theorem}

To prove \cref{thm:gl-finite-coinvariants-simple}, we use \cref{thm.steinbergcoinv}. The required input is a vanishing result for the homology of $\mr{GL}_n(\bF_q)$ with coefficients in the Steinberg module, and relative variations thereof. We shall provide these now.

For a finite field $\bF_q$ and $n\geq 2$, $\mr{St}(\bF_q^n) \otimes \bF_p$ is a nontrivial irreducible projective $\bF_p[\mr{GL}_n(\bF_q)]$-module \cite[Sections 2 and 3]{Humphreys} and thus we have
\[H_*(\mr{GL}_n(\bF_q);\mr{St}(\bF_q^n) \otimes \bF_p) = 0\]
for all $n \geq 2$. More precisely, the higher homology vanishes because it is projective, and the coinvariants vanish because, if not, they would provide a nontrivial quotient of a nontrivial irreducible module. 

The similar looking groups $H_d(\mr{GL}_n(\bF_q);\mr{St}^{E_1}(\bF_q^n) \otimes \bF_p)$ cannot vanish for all $d\geq 0$ and all $n \geq 2$: if they did, the $E_1$-homology of $\smash{\gR_{\bF_p}}$ would vanish for $n \geq 2$, implying that $\gR_{\bF_p} \simeq \gE_1(S^{0,1})$, which has the wrong homology. Consequently, by \cref{thm.steinbergcoinv}, the homology of the relative Steinberg modules $\strel{\bF_q^n}{W} \otimes \bF_p$ cannot vanish in all degrees for $n \geq 2$ either. However, we will show that their homology \emph{does} vanish in a range of degrees.

To do so, we will use the dual relative Tits building $\titsdualrel{P}{W}$ of \cref{def:rel-tits-dual}. Given a subspace $W \subset P$, this is the subposet of $\cT(P)$ of non-zero proper subspaces $V$ such that $V+W = P$. By taking duals, it is isomorphic to the ordinary relative Tits building $\titsrel{P^\vee}{W^\circ}$. The group $\mr{GL}(P,\text{ pres $W$, fix $P/W$})$ acts on $\tdualrel{P}{W}$, and under dualization this goes to action of the isomorphic group $\mr{GL}(P^\vee, \text{ fix $W^{\circ}$})$. As before, without loss of generality $W = \bF_q^w$.

\begin{lemma}\label{lem:finite-rel-steinberg} 
Let $n \geq 2$ and $0<w<n$, then 
	\[H_d\left(\mr{GL}(\bF_q^n, \text{ pres $\bF_q^w$, fix $\bF_q^n/\bF_q^w$}) ; \widetilde{H}_{w-1}(\tdualrel{\bF_q^n}{\bF_q^w}) \otimes \bF_p\right)=0\]
	for $d < r(p-1)-1$.
\end{lemma}

\begin{proof}
Let $L \subset \bF_q^w$ be a 1-dimensional subspace, and let $K$ be the subgroup of $\mr{GL}(\bF_q^n, \text{ pres $\bF_q^w$, fix $\bF_q^n/\bF_q^w$})$ given by 
	\[K \coloneqq \mr{GL}(\bF_q^n, \text{pres $L$, pres $\bF_q^w$, fix $\bF_q^n/\bF_q^w$}).\] 
The group $\mr{GL}(\bF_q^n, \text{pres $\bF_q^w$, fix $\bF_q^n/\bF_q^w$})$ acts transitively on the set of lines $\bP(\bF_q^w)$ in $W$ with $K$ the stabiliser of $L$. Thus this subgroup has index 
	\[\vert \bP(\bF_q^w) \vert = \frac{q^w-1}{q-1} \equiv 1 \mod p.\]
	It follows by a transfer argument that the map
	\[\begin{tikzcd}H_*\left(\mr{GL}(\bF_q^n \text{, pres $L$, pres $\bF_q^w$, fix $\bF_q^n/\bF_q^w$}); \widetilde{H}_{w-1}(\tdualrel{\bF_q^w}{\bF_q^n};\bF_p)\right) \dar{i_*} \\[-5pt] H_*\left(\mr{GL}(\bF_q^n, \text{ pres $\bF_q^w$, fix $\bF_q^n/\bF_q^w$}); \widetilde{H}_{w-1}(\tdualrel{\bF_q^w}{\bF_q^n};\bF_p)\right)\end{tikzcd}\]
	is surjective, so it is enough to show the vanishing of the source.
	
	By \cref{lem:p-decomposition} we have a weak equivalence
	\[\tdualrel{\bF_q^n}{\bF_q^w} \simeq \tdualrel{\bF_q^n/L}{\bF_q^w/L} * \tdualrel{\bF_q^n}{L},\]
	equivariant for the subgroup $K$ defined above. Thus we get an isomorphism
	\[\widetilde{H}_{w-1}(\tdualrel{\bF_q^n}{\bF_q^w}\;\bF_p) \cong \widetilde{H}_{w-2}(\tdualrel{\bF_q^n/L}{\bF_q^w/L};\bF_p) \otimes \widetilde{H}_0(\tdualrel{\bF_q^n}{L};\bF_p)\]
	of $\bF_p[K]$-modules. To deal with the source of $i_*$, we consider the subgroup of $K$ which acts trivially on $\bF_q^n/L$, i.e.\ 
	\[K' \coloneqq \mr{GL}(\bF_q^n,\text{ pres $L$, fix $\bF_q^n/L$}) \subset K.\]
	This is the kernel of the natural surjective homomorphism $K \to \mr{GL}(\bF_q^n/L)$ (see \cref{fig:kprime-matrix} for an example). It is isomorphic to the semidirect product $\bF^{n-1}_q \rtimes \bF_q^\times$.
	
	The group $K'$ acts trivially on $\widetilde{H}_{w-2}(\tdualrel{\bF_q^n/L}{\bF_q^w/L};\bF_p)$, so using the Serre spectral sequence it is enough to show vanishing of $H_i(K';\widetilde{H}_0(\tdualrel{\bF_q^n}{L};\bF_p))$ in the claimed range. Now $\tdualrel{\bF_q^n}{L}$ is the set of hyperplanes in $\bF_q^n$ complementary to $L$, and $K'$ acts transitively on these with stabiliser $\bF_q^\times$. Because this stabiliser is $\bF_p$-acyclic the exact sequence
	\[0 \lra \widetilde{H}_0(\tdualrel{\bF_q^n}{L};\bF_p) \lra \bF_p\{\titsdualrel{\bF_q^n}{L}\} \lra \bF_p \lra 0\]
implies that the connecting homomorphism
	\[H_{d+1}(\bF_q^{n-1};\bF_p)_{\bF_q^\times} = H_{d+1}(K';\bF_p) \overset{\sim}\lra H_d(K'; \widetilde{H}_0(\tdualrel{\bF_q^n}{L};\bF_p))\]
	is an isomorphism for $d \geq 0$. By the description of the $\bF_p$-homology of finitely-generated abelian $p$-groups, \cite[\S 11, esp.\ Lemma 16]{quillenfinite}, this vanishes for $d+1 < r(p-1)$.
\end{proof}

\begin{figure}
	\[\begin{bmatrix}
	\ast & \ast & \ast & \ast & \ast \\
	0 & 1 & 0 & 0 & 0 \\
	0 & 0 & 1 & 0 & 0 \\
	0 & 0 & 0 & 1 & 0 \\
	0 & 0 & 0 & 0 & 0 \\
	\end{bmatrix}\]
	\caption{An element of $K'$ when $n=5$ and $L = \langle e_1\rangle$.}
	\label{fig:kprime-matrix}
\end{figure}

\begin{remark}
Stripped down, \cref{lem:finite-rel-steinberg} implies that relative Steinberg coinvariants vanish for any field $\bF$ satisfying $\bF_{\bF^\times} = 0$. This is satisfied by all fields except $\bF_2$.
\end{remark}

\subsection{The proof of \cref{thm:main-fp}} In \cref{sec:einfty-algebra}, we defined a non-unital $E_\infty$-algebra $\gR_{\bF_p} \in \Alg_{E_\infty}(\cat{sMod}_{\bF_p}^\bN)$ satisfying
\[H_*(\gR_{\bF_p}(n)) \cong \begin{cases} 0 & \text{if $n = 0$,} \\
H_*(\mr{GL}_n(\bF_q);\bF_p) & \text{if $n>0$}, \end{cases}\]
and by \eqref{eqn:e1homology-ste1} we have $H_{d-n+1}(\mr{GL}_n(\bF_q);\mr{St}^{E_1}(\bF_q^n) \otimes \bF_p) \cong H^{E_1}_{n,d}(\gR_{\bF_p})$. Thus \cref{thm:gl-finite-coinvariants-simple} implies that as long as $n \geq 2$ and $r(p-1)-2 \geq 0$ (i.e.\ $q \neq 2$) we have $H^{E_1}_{n,d}(\gR_{\bF_p})=0$ in degrees $d < n+r(p-1)-2$. 

Let $\gN_{\bF_p} \in \cat{sMod}_{\bF_p}^\bN$ be given by
\[\gN_{\bF_p}(n) \coloneqq \begin{cases}
0 & \text{if $n=0$,}\\
\bF_p & \text{if $n > 0$.}
\end{cases}\]
This has the structure of a non-unital commutative monoid in $\cat{sMod}_{\bF_p}^\bN$, so in particular of a non-unital $E_\infty$-algebra. The augmentations $\epsilon \colon H_0(\mr{GL}_n(\bF_q);\bF_p) \to \bF_p$ assemble to a map of $E_\infty$-algebras
\[f \colon \gR_{\bF_p} \lra \gN_{\bF_p}.\]
We shall use $H^{E_\infty}_{n,d}(f)$ to denote the relative $E_\infty$-homology of the map $f$, that is, the homology of the homotopy cofiber of the induced map $Q^{E_\infty}_\bL(\gR_{\bF_p}) \to Q^{E_\infty}_\bL(\gN_{\bF_p})$.

\begin{lemma}$H^{E_\infty}_{n,d}(f) = 0$ for $d<\max(r(2p-3)+1,n+r(p-1)-1)$.\end{lemma}

\begin{proof}
Let $S^{1,0}_{\bF_p}$ be the object of $\cat{sMod}^\bN_{\bF_p}$ given by 
\[S^{1,0}_{\bF_p}(n) = \begin{cases} \bF_p[D^1/\partial D^1] & \text{if $n=1$,} \\
0 & \text{otherwise,}\end{cases}\]
where $D^1/\partial D^1$ is considered as a \emph{based} simplicial set. Then on the one hand, $\gN_{\bF_p}$ is weakly equivalent to $\gE_1(S^{1,0}_{\bF_p})$ as an $E_1$-algebra, so its $E_1$-homology vanishes in all bidegrees except $(1,0)$. On the other hand, it is a consequence of \cref{thm:gl-finite-coinvariants-simple} that $H^{E_1}_{n,d}(\gR_{\bF_p})=0$ when $d < n+r(p-1)-2$ and $n \geq 2$.

By \cite[Lemma A.1]{FriedlanderParshall}, $H_{n,d}(\gR_{\bF_p})=0$ for $0 < d < r(2p-3)$, so $H_{n,d}(f) = 0$ for $d < r(2p-3)+1$. Using the Hurewicz theorem for $E_1$-homology (Corollary $E_k$.11.14) and the long exact sequence on $E_1$-homology (equation (10.1) in Section $E_k$.10.1.6), we conclude that $H^{E_1}_{n,d}(f) = 0$ for $d<\max(r(2p-3)+1,n+r(p-1)-1)$.

We wish to deduce the same vanishing line for the $E_\infty$-homology of $f$. The functions 
\[\rho(n) \coloneqq n \qquad \text{and} \qquad 	\sigma(n) \coloneqq \max(r(2p-3)+1,n+r(p-1)-1)\]
can be considered as functors $\bN \to [-\infty,\infty]_{\geq}$, i.e.~an \emph{abstract connectivity} as discussed in Section $E_k$.11.1. Letting $\ast$ denote the Day convolution of functors as in formula (11.1) of Section $E_k$.11.1, these satisfy $\rho \ast \rho \geq \rho$ and $\rho \ast \sigma \geq \sigma$. We now apply Proposition $E_k$.14.5 to the map $f$. Its hypotheses are that $\smash{H^{E_1}_{n,d}}(\gR_{\bF_p}) = 0 =\smash{H^{E_1}_{n,d}}(\gN_{\bF_p})$ for $d<\rho(n)-1$ and that $\smash{H^{E_1}_{n,d}}(f)=0$ for $d<\sigma(n)$, which were verified above, and its conclusion is that for all $k \geq 1$ we have that $\smash{H_{n,d}^{E_k}}(f)=0$ for $d < \sigma(n)$. Letting $k \to \infty$, we conclude that $H^{E_\infty}_{n,d}(f) = 0$ for $d < \sigma(n)$, i.e.\ for $d<\max(r(2p-3)+1,n+r(p-1)-1)$, as required.
\end{proof}

From $\gR_{\bF_p}$ we may construct a unital strictly associative algebra $\overline{\gR}_{\bF_p}$, which is equivalent to the unitalization $\gR_{\bF_p}^+$. Picking a representative $\sigma \colon \smash{S^{1,0}_{\bF_p}} \to \smash{\gR_{\bF_p}}$, we get a corresponding map
\[\sigma \cdot - \colon S^{1,0}_{\bF_p} \otimes \overline{\gR}_{\bF_p} \lra \overline{\gR}_{\bF_p}\]
and we may form its homotopy cofibre $\overline{\gR}_{\bF_p}/\sigma$. The operation $S^{1,0}_{\bF_p} \otimes -$ shifts the rank grading by $1$, and the homology groups of this homotopy cofibre may be identified with the relative homology groups of the stabilisation map
\[H_{n,d}(\overline{\gR}_{\bF_p}/\sigma) \cong H_{d}(\mr{GL}_{n}(\bF_q),\mr{GL}_{n-1}(\bF_q);\bF_p).\]
To prove \cref{thm:main-fp} it thus suffices to prove the following vanishing range for the groups $H_{n,d}(\overline{\gR}_{\bF_p}/\sigma)$.

\begin{theorem}\label{thm:VanishingNot2}
$H_{n,d}(\overline{\gR}_{\bF_p}/\sigma) = 0$ for $d < \max(r(2p-3),n+r(p-1)-2)$.
\end{theorem}

\begin{proof} The strategy shall be to understand which $E_\infty$-cells need to be attached to $\gR_{\bF_p}$ to build $\gN_{\bF_p}$. We of course know the homology of $\gN_{\bF_p}/\sigma$, and by understanding the effect of the $E_\infty$-cells that have been added to $\gR_{\bF_p}$ we shall draw the stated conclusion about $\gR_{\bF_p}/\sigma$.
		
By Theorem $E_k$.11.21 we can extend the map $f$ to a diagram of $E_\infty$-algebras
\[\begin{tikzcd} \gR_{\bF_p} \rar{f} \dar{g} & \gN_{\bF_p} \\
\gC \arrow{ru}[swap]{\simeq}, & \end{tikzcd}\]
where $g$ is a relative CW object obtained by attaching $E_\infty$-cells of bidegrees $(n,d)$ satisfying $d \geq \max(r(2p-3)+1,n+r(p-1)-1)$.

The relative CW object $\gC$ is equipped with a skeletal filtration $\mr{sk}(\gC)$ and Theorem $E_k$.10.10 constructs from the filtered object $\overline{\mr{sk}(\gC)}$ a spectral sequence
\[E^1_{n,p,q}(\overline{\gN}_{\bF_p}) = H_{n,p+q,p}\left(\overline{0_* \gR_{\bF_p} \vee^{E_\infty} \bigvee^{E_\infty}_{i \in I} S^{n_i,d_i,d_i}_{\bF_p} x_i}\right) \Longrightarrow H_{n,p+q}(\overline{\gN}_{\bF_p}),\]
where $(n_i,d_i)$ satisfy $n_i \geq 1$ and $d_i \geq \max(r(2p-3)+1,n_i +r(p-1)-1)$. The $E^1$-page can be expressed quite concretely. By the calculation of the homology of free $E^+_\infty$-spaces there is a monad $W_\infty$ on the category of (possibly multiply) graded vector spaces such that $H_*(\gE^+_\infty(X);\bF_p) = W_\infty(H_{*}(X;\bF_p))$. It follows that the homology of any $E^+_\infty$-algebra has the structure of a $W_\infty$-algebra. Explicitly, $W_\infty$ is given by the ``free unstable algebra over the Dyer--Lashof algebra", and the elements of $W_\infty(V)$ are products of certain Dyer--Lashof operations $Q^I$ applied to elements of $V$; a detailed exposition is given in Section $E_k$.16. Furthermore, the homology of a coproduct of $E^+_\infty$-algebras is given by the tensor product of their homologies (Corollary $E_k$.16.5). Using this, we may express the $E^1$-page as
\[E^1_{*,*,*}(\overline{\gN}_{\bF_p}) = H_{*,0,*}(\overline{\gR}_{\bF_p}) \otimes \bigotimes_{i \in I} W_\infty(\bF_p \{ x_i\}[n_i, d_i, d_i]).\]
In this expression, $H_{*,0,*}(\overline{\gR}_{\bF_p})$ means we consider the bigraded group $H_{*,*}(\overline{\gR}_{\bF_p})$ as trigraded by concentrating it in middle grading $0$. We will now analyse the degrees in which classes from the second tensor factor may lie, using the description of $W_\infty(\bF_p \{ x_i\}[n_i, d_i, d_i])$ as products of classes $Q^I (x_i)$ for certain $I$'s. The proof requires detailed information about which Dyer--Lashof operations occur, and how they affect degrees; we refer to Section $E_k$.16 for all the required background.

\vspace{.5em}

\noindent\textbf{Claim.} All elements of $W_\infty(\bF_p \{x_i\}[n_i,d_i,d_i])$ lie in tridegrees $(n,p,q)$ satisfying $p+q \geq \max(r(2p-3)+1,n+r(p-1)-1)$, and so all elements in their tensor products do too. 

\begin{proof}[Proof of claim]

We may forget the filtration degree $q$ for the purposes of this claim, and work with the bigrading $(n,d) = (n,p+q)$. The generator $x_i$ lies in bidegree $(n_i, d_i)$, so by the inequality $d_i \geq \max(r(2p-3)+1,n_i +r(p-1)-1)$ it satisfies the required condition. Let us investigate how the operations $Q^I$ affect bidegrees.

\vspace{1ex}

\noindent\textbf{The case $p$ odd}. Applying a Dyer--Lashof operation $Q^s$ changes bidegrees by $(n,d) \mapsto (pn,d+2s(p-1))$ with $s \geq d$, and $\beta Q^s$ by $(n,d) \mapsto (pn,d+2s(p-1)-1)$ with $s>d$. This does not decrease $d$. We claim it also preserves the property of satisfying the inequality $d \geq n+r(p-1)-1$. It suffices to verify this in the following cases, as these are the smallest $s$ such that $Q^s(x)$ can be non-zero:
\begin{enumerate}[(i)]
		\item the element $Q^{k}(x)$ if $|x|=(n,2k-\epsilon)$ with $\epsilon=0,1$: if $2k-\epsilon \geq n+r(p-1)-1$ then 
\begin{align*}
	(2k-\epsilon)+2k(p-1) &= p(2k-\epsilon)+\epsilon(p-1) \\
	&\geq p(2k-\epsilon) \\
	&\geq p(n+r(p-1)-1) \\
	&\geq pn +r(p-1)-1 \quad \text{as $r(p-1)-1 > 0$.}
\end{align*}

		\item the element $\beta Q^{k+1}(x)$ if $|x|=(n,2k+2-\epsilon)$ with $\epsilon=0,1$: if $2k+2-\epsilon \geq n+r(p-1)-1$ then similarly \
		\begin{align*}
	(2k+2-\epsilon)+2(k+1)(p-1)-1 &= p(2k+2-\epsilon)+\epsilon(p-1)-1 \\
	&\geq p(2k+2-\epsilon)-1\\
	&\geq p(n+r(p-1)-1)-1\\
	&\geq pn+r(p-1)-1.
\end{align*}
\end{enumerate}

\vspace{1ex}

\noindent\textbf{The case $p=2$}. The operation $Q^s$ changes bidegrees by $(n,d) \mapsto (2n,d+s)$ for $s>d$. Again it does not decrease $d$ and we claim it preserves the property of satisfying the inequality $d \geq n+r-1$. As above it suffices to verify this for the smallest $s$ which can be non-zero, $Q^{d}(x)$ for $|x|=(n,d)$ (which equals squaring). Assuming $d \geq n+r-1$, we get 
	\[2d \geq 2(n+r-1) \geq 2n+r-1.\]

\vspace{1ex}

It follows that the bidegrees of all classes $Q^I x_i$ satisfy the required condition, but then all products of such classes do too.
\end{proof}

From the quotient of the filtered object $\overline{\mr{sk}(\gC)}$ by $\sigma$, we obtain the spectral sequence
\[E^1_{n,p,q}(\overline{\gN}_{\bF_p}/\sigma) = {H}_{n,p+q,p}\left(\left(\overline{0_* \gR_{\bF_p} \vee^{E_\infty} \bigvee^{E_\infty}_{i \in I} S^{n_i,d_i,d_i}_{\bF_p} x_i}\right)/\sigma\right) \Rightarrow H_{n,p+q}(\overline{\gN}_{\bF_p}/\sigma),\]
which converges to $0$ for $(n,p+q) \neq (0,0)$ and is a module over the spectral sequence $\{E^r_{n,p,q}\}$. As above we may express its $E^1$-page as
\[E^1_{*,*,*}(\overline{\gN}_{\bF_p}/\sigma) = H_{*,0,*}(\overline{\gR}_{\bF_p}/\sigma) \otimes \bigotimes_{i \in I} W_\infty(\bF_p \{ x_i\}[n_i, d_i, d_i]).\]

Our goal is to show that $H_{n,d}(\overline{\gR}_{\bF_p}/\sigma) = 0$ for $d < \max(r(2p-3),n+r(p-1)-2)$, with the exception of $H_{0,0}(\overline{\gR}_{\bF_p}/\sigma) = \bF_p$. For $d < r(2p-3)$ this follows from the theorem of Friedlander--Parshall \cite[Lemma A.1]{FriedlanderParshall}.

For $d \geq r(2p-3)$ we give a proof by contradiction. Suppose that $d$ is the smallest non-zero degree in which there is an $n$ such that $H_{n,d}(\overline{\gR}_{\bF_p}/\sigma) \neq 0$ for $d < n+r(p-1)-2$. Then $E^1_{n,0,d}(\overline{\gN}_{\bF_p}/\sigma) \neq 0$, and since the spectral sequence $\{E^r_{n,p,q}(\overline{\gN}_{\bF_p}/\sigma)\}$ has to converge to 0 with the exception of tridegree $(0,0,0)$ and the $d^r$-differential has tridegree  $(0,-r,r-1)$, it must be that $E^1_{n, \ell,d+1-\ell}(\overline{\gN}_{\bF_p}/\sigma) \neq 0$ for some $\ell \geq 1$. The group $E^1_{n,\ell,d+1-\ell}(\overline{\gN}_{\bF_p}/\sigma)$ is spanned by products of an element of $H_{n',0,d'}(\overline{\gR}_{\bF_p}/\sigma)$ with an element of the tensor product of $W_\infty$-algebras of tridegree $(n'',\ell,d'')$, which must satisfy $d'' \geq \ell \geq 1$.

If $(n',d') = (0,0)$ then we have $(n'',\ell,d'') = (n,\ell,d+1)$, which is impossible because then $d'' = d+1 < n+r(p-1)-1 = n''+r(p-1)-1$, but all elements in the tensor product of $W_\infty$-algebras satisfy $d'' \geq n''+r(p-1)-1$. 

On the other hand if $(n',d') \neq (0,0)$ then, as $d'' \geq 1$ we have $d' < d$ and so as $d$ was assumed to be minimal such that there is an $n$ such that $H_{n,d}(\overline{\gR}_{\bF_p}/\sigma) \neq 0$ for $d < n+r(p-1)-2$, we must have that $d' \geq n'+r(p-1)-2$. As $d'+d'' < n'+n'' + r(p-1)-2$ it follows that $d'' < n''$, but this contradicts $d'' \geq n'' + r(p-1)-1$.
\end{proof}

\subsection{Sharpness} \label{sec:fp-sharpness} In this section we explain in what sense \cref{thm:main-fp} and its consequences in \cref{cor:vanishing} are optimal.

We have already pointed out that $H_{n,d}(\gR_{\bF_p})=0$ for $0 < d < r(2p-3)$, by \cite[Lemma A.1]{FriedlanderParshall}, which can be added to \cref{cor:vanishing} to obtain stronger vanishing results. Just beyond this range, Sprehn \cite[Theorem 1]{SprehnPaper} has shown that \[H_{n, r(2p-3)}(\gR_{\bF_p}) \neq 0 \qquad \text{ for $2 \leq n \leq p$.}\]
When $r=1$ these groups are one-dimensional and $H_{n,2p-3}(\gR_{\bF_p})$ vanishes for $n > p$ \cite[Theorem 3.1]{SprehnThesis}. Thus it is known that the last rank in which $H_{n,2p-3}(\gR_{\bF_p})$ does not vanish is $p$, and our \cref{cor:vanishing} says that $H_{n,2p-3}(\gR_{\bF_p})$ vanishes for $n+p-2>2p-2$ or equivalently $n>p$. Thus for $r=1$, in \cref{thm:main-fp} neither the offset nor the slope can be improved while keeping the other fixed.

This also shows our vanishing results for the homology of $\mr{St}^{E_1}(\bF_p^n) \otimes \bF_p$ are sharp. For the stabilisation of the non-zero class in $H_{p,2p-3}(\gR_{\bF_p})$ to vanish in $H_{p+1,2p-3}(\gR_{\bF_p})$, there must be an $E_1$-cell in bidegree $(p+1,2p-2)$ and hence $\smash{H^{E_1}_{p+1,2p-2}}(\gR_{\bF_p}) \neq 0$. Thus $H_{p-2}(\mr{GL}_{p+1}(\bF_p);\smash{\mr{St}^{E_1}}(\bF_p^{p+1})\otimes \bF_p) \neq 0$, which is just outside our vanishing line of $d<p-2$ for $H_d(\mr{GL}_{n}(\bF_p);\mr{St}^{E_1}(\bF_p^{n})\otimes \bF_p)$.

\begin{remark}When $r=1$, following Sprehn's proof that $H_{p, 2p-3}(\gR_{\bF_p}) = \bF_p$, one can show that this group is generated by $\beta Q^1(\sigma)$. This implies that $\smash{H_{p, 2p-3}^{E_\infty}}(\gR_{\bF_p}) = 0$, which is \emph{not} a consequence of our vanishing line $\smash{H_{n,d}^{E_1}}(\gR_{\bF_p})=0$ for $d < n+(p-1)-2$.
\end{remark}

\begin{remark}
Further unstable classes in bidegrees $(n,d)=(p^N, r(2p^N-2p^{N-1}-1))$ for $N \geq 2$ were produced by Lahtinen--Sprehn \cite{Lahtinen-Sprehn}. These classes are outside the vanishing range of \cref{cor:vanishing} because $r(2p^N-2p^{N-1}-1) \geq p^N+r(p-1)-3$. For example, taking $r=1$, $p=3$, and $N=2$, their class is in bidegree $(9,11)$ and our vanishing range for $n=9$ is $d<9$. In general their classes lie above a line of slope $\frac{2rp-2r}{p}$.
\end{remark}

\section{General linear group of the field $\bF_2$} 

The proof of \cref{thm:main-f2} for $\mr{GL}_n(\bF_2)$ is substantially different. It only uses that $S^{E_1}(P)$ is $(\dim(P)-2)$-spherical and low rank computations. It does not involve any computations of homology with coefficients in the $E_1$-Steinberg module.

\subsection{Low rank computations} \label{sec:f2-low-rank} We will do some computations in low rank and low homological degree of the homology groups $H_{n,d}(\gR_{\bF_2}) \cong H_d(\mr{GL}_n(\bF_2);\bF_2)$, taking particular care to describe these in terms of Dyer-Lashof operations. Before doing so we give a variation of a result of Quillen \cite[Corollary 1]{quillenfinite}.

\begin{lemma}The iterated stabilisation maps 
	\begin{align*}\sigma^{k+1} \colon \mr{GL}_{n-1}(\bF_{2}) &\longrightarrow \mr{GL}_{k+n}(\bF_{2}) \\
	\sigma^{k} \colon \mr{GL}_n(\bF_{2}) &\longrightarrow \mr{GL}_{k+n}(\bF_{2})\end{align*}
	have the same image on $\bF_2$-homology in degrees $d<k+1$.\end{lemma}

\begin{proof}
	The $2$-Sylow subgroup of $\mr{GL}_n(\bF_2)$ is given by the group $B_n(\bF_2) \subset \mr{GL}_n(\bF_2)$ of upper triangular matrices. We shall define a larger group $H$ containing this group. Choosing an inclusion $\bF_2 \hookrightarrow \bF_{2^{k+1}}$ we may consider $\bF_2$ as a subfield of $\bF_{2^{k+1}}$. Then $H$ is defined to be the subgroup of $B_n(\bF_{2^{k+1}})$ of matrices such that all entries below the first row lie in $\bF_2 \subset \bF_{2^{k+1}}$, that is
	\[H \coloneqq \left[\begin{array}{@{}c|c@{}}\bF_{2^{k+1}}^\times & (\bF_{2^{k+1}})^{n-1} \\ \hline
	0 & B_{n-1}(\bF_2)\end{array}\right].\]
	This contains $B_n(\bF_2)$ as the subgroup of matrices all of whose entries lie in $\bF_2 \subset \bF_{2^{k+1}}$.
	
	By considering $\bF_{2^{k+1}}$ as a $(k+1)$-dimensional $\bF_2$-vector space, once we pick a basis we get a homomorphism $\phi \colon H \to \mr{GL}_{k+n}(\bF_2)$. If we pick this basis such that $1 \in \bF_{2^{k+1}}$ is the last basis vector, then because $1$ is the only unit in $\bF_2$ the following diagram commutes:
	\begin{equation}\label{eq:stabdiag} \begin{tikzcd} B_{n-1}(\bF_2) \rar \dar[hook] & B_n(\bF_2) \rar \dar[hook] & H \dar[hook]{\phi} \\[-2pt]
	\mr{GL}_{n-1}(\bF_2) \rar{\sigma} & \mr{GL}_{n}(\bF_2) \rar{\sigma^{k}} & \mr{GL}_{n+k}(\bF_2). \end{tikzcd} \end{equation}
	
	We claim that the inclusion $B_{n-1}(\bF_2) \hookrightarrow B_n(\bF_2) \hookrightarrow H$ is an isomorphism on $\bF_2$-homology in degrees $d<k+1$. To do so, we use the commutative diagram
	\[\begin{tikzcd} \{e\} \dar \rar & \bF_{2^{k+1}}^\times \ltimes \bF_{2^{k+1}}^{n-1} \dar \\[-3pt]
	B_{n-1}(\bF_2) \rar[hook] \dar[equals] & H \dar[two heads] \\[-3pt]
	B_{n-1}(\bF_2) \rar[equals] & B_{n-1}(\bF_2).\end{tikzcd}\]
	Thus the claim follows by comparing the Serre spectral sequences for the two vertical sequences, as soon as we recall that the $\bF_2$-homology of the group $\bF_{2^{k+1}}^\times \ltimes \bF_{2^{k+1}}^{n-1}$ vanishes for $d<k+1$ by \cite[Theorem 6]{quillenfinite}.
	
	Thus in \eqref{eq:stabdiag}, the first two vertical maps are surjective on homology with $\bF_2$-coefficients, while the composite of the top two arrows is an $\bF_2$-homology isomorphism in degrees $d<k+1$.	The result follows by a diagram chase.
\end{proof}

\begin{corollary}\label{cor.stabilisation-vanishing} $\sigma^k \cdot - \colon H_{n,d}(\gR_{\bF_2}) \to H_{n+k,d}(\gR_{\bF_2})$ vanishes for $0<d<k+1$.
\end{corollary}

\begin{proof}This is a proof by contradiction. Let $x$ be an element of degree $0<d<k+1$ and of minimal rank $n$ such that $\sigma^k x \neq 0$. As $d>0$ we must have $n \geq 1$, as $H_{0,d}(\gR_{\bF_2})=0$. By the previous lemma there is an element $y$ in rank $n-1 \geq 0$ such that $\sigma^k x = \sigma^{k+1} y$, but $\sigma^k y = 0$ as $n$ was minimal.
\end{proof}

\begin{figure}[h]
	\begin{tikzpicture}
	\begin{scope}
	\draw (-1,0)--(9.5,0);
	\draw (0,-1) -- (0,2.5);
	
	\foreach \s in {0,...,2}
	{
		\draw [dotted] (-.5,\s)--(9.5,\s);
		\node [fill=white] at (-.25,\s) [left] {\tiny $\s$};
	}
	
	\foreach \s in {1,...,4}
	{
		\draw [dotted] ({2*\s},-0.5)--({2*\s},2.5);
		\node [fill=white] at ({2*\s},-.5) {\tiny $\s$};
	}
	\node [fill=white] at (0,-.5) {\tiny 0};
	\node [fill=white] at ({2*1},0) {$\sigma$};
	\node [fill=white] at ({2*2},0) {$\sigma^2$};
	\node [fill=white] at ({2*3},0) {$\sigma^3$};
	\node [fill=white] at ({2*4},0) {$\sigma^4$};
	\node [fill=white] at ({2*2},1) {$Q^1(\sigma)$};
	\node [fill=white] at ({2*2},2) {$Q^2(\sigma)$};
	%	\node [fill=white] at ({2*2},3) {$Q^3$};
	\node [fill=white] at ({2*3},2) {$\sigma Q^2(\sigma)$};
	%	\node [fill=white] at ({2*3},3) {$b$, $c$};
	\node [fill=white] at ({2*4},2) {$(Q^1(\sigma))^2$};
	%	\node [fill=white] at ({2*4},3) {$Q^2Q^1$, $Q^{2,1}$};	

	\node [fill=white] at (-.5,-.5) {$\nicefrac{d}{g}$};
	\end{scope}
	\end{tikzpicture}
	\caption{The additive generators of the $\bF_2$-homology of $\mr{GL}_n(\bF_2)$ for low degree and low rank.}
	\label{fig:f2lowrank}
\end{figure}

\begin{lemma}\label{lem:f2-low-rank}See \cref{fig:f2lowrank} for a depiction of the following computations of $H_{n,d}(\gR_{\bF_2}) = H_d(\mr{GL}_n(\bF_2);\bF_2)$:
	\begin{enumerate}[(i)]
		\item \label{enum:f2-1d} $H_{1,d}(\gR_{\bF_2}) = 0$ for all $d >0$.
		\item \label{enum:f2-2d} $H_{2,d}(\gR_{\bF_2}) = \bF_2$ for all $d>0$, generated by $Q^d(\sigma)$.
		\item \label{enum:f2-3d} $H_{3,1}(\gR_{\bF_2}) = 0$, and $H_{3,2}(\gR_{\bF_2}) = \bF_2$, generated by $\sigma Q^2(\sigma)$.
		\item \label{enum:f2-4d} $H_{4,1}(\gR_{\bF_2}) = 0$, and $H_{4,2}(\gR_{\bF_2}) = \bF_2$, generated by $(Q^1(\sigma))^2$.
		\item \label{enum:f2-stab} The map $\sigma \cdot - \colon H_{3,2}(\gR_{\bF_2}) \to H_{4,2}(\gR_{\bF_2})$ is zero. 
	\end{enumerate}
\end{lemma}

\begin{proof}Item (\ref{enum:f2-1d}) follows from the observation that $\mr{GL}_1(\bF_2)$ is the trivial group.
	
	For (\ref{enum:f2-2d}), we use that $\mr{GL}_2(\bF_2)$ is isomorphic to the symmetric group $\mathfrak{S}_3$, by acting on the set $\bF^2_2 \setminus \{0\}$ of non-zero vectors, so in particular the inclusion $\fS_2 \to \mr{GL}_2(\bF_2)$ of permutation matrices is an isomorphism on $\bF_2$-homology. The inclusion of permutation matrices forms part of an $E_\infty$-map, and the $\bF_2$-homology of $\fS_2$ is well-known to be generated by $Q^d(\sigma)$ in degree $d$.
	
	 For (\ref{enum:f2-3d}), the cohomology ring of $\mr{GL}_3(\bF_2)$ with $\bZ$-coefficients was computed in \cite{tezukayagita}, and from it the $\bF_2$-homology is easily deduced using the Universal Coefficients Theorem. 
	 
	 To see that the generator of $H_2$ is $\sigma Q^2(\sigma)$, we use Stiefel--Whitney classes. The action of $\mr{GL}_3(\bF_2)$ on the set $\bF_2^3\setminus \{0\}$ gives a 7-dimensional real $\mr{GL}_3(\bF_2)$-representation $\rho_3$. Composing with the inclusion $\iota_2 \colon \fS_2 \hookrightarrow \mr{GL}_2(\bF_2) \hookrightarrow \mr{GL}_3(\bF_2)$ gives a real $\fS_2$-representation, which is the sum of 2 copies of the sign representation and 5 copies of the trivial representation. Thus its total Stiefel--Whitney class is $w(\rho_3 \circ \iota_2) = (1+w_1(\mr{sign}))^2 = 1+w_1(\mr{sign})^2$. In the usual identification $H^*(\fS_2;\bF_2) = \bF_2[x]$ the Stiefel--Whitney class $w_1(\mr{sign})$ is $x$. As $Q^2(\sigma) \in H_2(\fS_2;\bF_2)$ is non-zero, it must pair to 1 with $x^2 = w_1(\mr{sign})^2$, from which it follows that $\sigma Q^2(\sigma) \in H_2(\mr{GL}_3(\bF_2);\bF_2)$ pairs to 1 with $w_2(\rho_3)$, and is hence non-zero.
 
	  For (\ref{enum:f2-4d}), we use the exceptional isomorphism $A_8 \cong \mr{GL}_4(\bF_2)$ and the computation of the cohomology ring of $A_8$ with $\bF_2$-coefficients by Adem--Milgram \cite[Theorem VI.6.5]{adem-milgram}. To describe the generator in degree 2, we again use Stiefel--Whitney classes. Acting on the set of non-zero vectors $\bF_2^4 \setminus \{0\}$ gives a 15-dimensional real $\mr{GL}_4(\bF_2)$-representation $\rho_4$. By composing with the block permutation matrix inclusion $\iota_{2 \times 2} \colon \fS_2 \times \fS_2 \to \mr{GL}_4(\bF_2)$ we obtain a 15-dimensional real $\fS_2 \times \fS_2$-representation $\rho_4 \circ \iota_{2 \times 2}$. This splits as one copy of $\mr{sign} \boxtimes \mr{sign}$, $3$ copies of $\mr{triv} \boxtimes \mr{sign}$, $3$ copies of $\mr{sign} \boxtimes \mr{triv}$, and 8 copies of $\mr{triv} \boxtimes \mr{triv}$, where $\boxtimes$ denotes the external tensor product of representations. Thus, under the identification $H^*(\fS_2 \times \fS_2;\bF_2) = \bF_2[x_1, x_2]$ given by the K{\"u}nneth theorem,  its total Stiefel--Whitney class is
\begin{align*}w(\rho_4 \circ \iota_{2 \times 2}) &= (1+x_1 + x_2) \cdot (1+x_1)^3 \cdot (1+x_2)^3 \\
&= 1 + x_1 x_2 + \text{ classes of degree $\geq 3$.}
\end{align*}	
	  The cohomology class $x_1 x_2$ pairs to 1 with the homology class $Q^1(\sigma) \times Q^1(\sigma) \in H_2(\fS_2 \times \fS_2;\bF_2)$, which under $\iota_{2 \times 2}$ maps to $Q^1(\sigma)^2$.
  
 	For item \eqref{enum:f2-stab}, it is a consequence of \cref{cor.stabilisation-vanishing} that stabilising twice vanishes in degrees $< 3$, so in particular $\sigma^2 Q^2(\sigma) = 0$. An alternative proof may extracted from \cite{vanderkallenschur}.
\end{proof}

\subsection{The proof of \cref{thm:main-f2}} 

The canonical generator $\sigma$ of $H_0(\mr{GL}_1(\bF_2);\bF_2)$ may be represented by a map $\sigma \colon S^{1,0} \to \gR_{\bF_2}$.  We further pick a $0$- and a $1$-simplex in $\bF_2[\cC_\infty(2)] \in \cat{sMod}_{\bF_2}$ representing the product and the operation $Q^1(-)$. These may be used to first produce a representative $S^{1,1} \to \gR_{\bF_2}$ for $Q^1(\sigma)$, and then one for $\sigma Q^1(\sigma)$. As we proved that $\sigma Q^1(\sigma) = 0$, we may pick a null-homotopy $D^{3,2} \to \gR_{\bF_2}$ for the latter. Out of this data, we may build a map in $\cat{Alg}_{E_\infty}(\cat{sMod}_{\bF_2}^\bN)$
\begin{align*}
      f \colon \gA   &\longrightarrow \gR_{\bF_2}\\
      \text{with } \gA & \coloneqq \gE_\infty(S^{1,0}_{\bF_2} \sigma) \cup^{E_\infty}_{\sigma Q^1(\sigma)} D^{3,2}_{\bF_2} \tau.
\end{align*}
By \cref{lem:f2-low-rank}, this satisfies $H_{n,d}(f) = 0$ for $n \leq 3$ and $d \leq 2$. Using the Hurewicz theorem in $E_\infty$-homology (Corollary $E_k$.11.14) we see that $\smash{H_{n,d}^{E_\infty}}(f) = 0$ for $n \leq 3$ and $d \leq 2$. On the other hand, we know that $\smash{H_{n,d}^{E_\infty}}(\gA)$ is concentrated in bidegrees $(n,d) = (1,0),(3,2)$, while $H_{n,d}^{E_\infty}(\gR_{\bF_2}) = 0$ for $d<n-1$. The conclusion is that $\smash{H_{n,d}^{E_\infty}(f) = 0}$ for $d < n$ if $n \leq 3$, and for $d < n-1$ if $n > 3$.

As before, from $\gR_{\bF_2}$ we may construct a unital strictly associative algebra $\smash{\overline{\gR}_{\bF_2}} \simeq \smash{\gR_{\bF_2}^+}$. From $\sigma$ we get a corresponding map $\sigma \cdot - \colon \smash{\overline{\gR}_{\bF_2}} \to \smash{\overline{\gR}_{\bF_2}}$ whose homotopy cofibre $\overline{\gR}_{\bF_2}/\sigma$ satisfies
	\[H_{n,d}(\overline{\gR}_{\bF_2}/\sigma) \cong H_{d}(\mr{GL}_{n}(\bF_2),\mr{GL}_{n-1}(\bF_2);\bF_2),\]
and to prove \cref{thm:main-f2} it thus suffices to prove a vanishing range for $H_{n,d}(\overline{\gR}_{\bF_2}/\sigma)$.

\begin{proposition}$H_{n,d}(\overline{\gR}_{\bF_2}/\sigma) = 0$ for $d<\frac{2(n-1)}{3}$.\end{proposition}

\begin{proof} We apply the CW approximation theorem (Theorem $E_k$.11.21) to extend the map $f$ to a commutative diagram
\[\begin{tikzcd}\gA \rar{f} \dar[swap]{g} & \gR_{\bF_2} \\
\gC \arrow{ru}[swap]{\simeq} & \end{tikzcd}\]
with $g \colon \gA \to \gC$ a relative CW approximation with only cells of bidegrees $(n,d)$ satisfying $d \geq n$ if $n \leq 3$ and $d \geq n-1$ if $n>3$.
The relative CW object $\gC$ is equipped a skeletal filtration $\mr{sk}(\gC)$ which is constantly $\gA$ and using Theorem $E_k$.10.10 we obtain from the filtered object $\overline{\mr{sk}(\gC)}/\sigma$ a corresponding skeletal spectral sequence
\begin{equation}\label{eqn:ga-skel-ss} 
	E^1_{n,p,q} = H_{n,p+q,p}\left (\left(\overline{0_* \gA \vee^{E_\infty} \bigvee^{E_\infty}_{i \in I} S^{n_i,d_i,d_i}_{\bF_2} x_i} \right)/\sigma \right) \Longrightarrow H_{n,p+q}(\overline{\gR}_{\bF_2}/\sigma),
\end{equation}
with $n_i$ and $d_i$ satisfying the bounds given above. Just as in the proof of \cref{thm:VanishingNot2} we identify the $E^1$-page as
	\[E^1_{*,*,*} = H_{*,0,*}(\overline{\gA}/\sigma) \otimes \bigotimes_{i \in I} W_\infty(\bF_2 \{ x_i\}[n_i, d_i, d_i]).\]
Here $H_{*,0,*}(\overline{\gA}/\sigma)$ denotes the bigraded object $H_{*,*}(\overline{\gA}/\sigma)$ made trigraded by concentrating it in middle grading $0$. Since $d_i \geq n_i$ if $n_i \leq 3$ and $d_i \geq n_i-1$ if $n_i>3$, for all $i \in I$, just as the proof of \cref{thm:VanishingNot2} all non-zero elements of $\bigotimes_{i \in I} W_\infty(\bF_2 \{ x_i\}[n_i, d_i, d_i])$ satisfy these inequalities too. Thus it remains to show that $H_{n,d}(\overline{\gA}/\sigma)=0$ for $d<\frac{2(n-1)}{3}$.

Computing the homology of $\overline{\gA}/\sigma$ outright seems to be complicated, so we shall be content with computing enough of it to establish the required vanishing. To do so, we first investigate the spectral sequence for the cell-attachment filtration on $\overline{\gA}$, as described in Corollary $E_k$.10.17. This spectral sequence interacts well with Dyer--Lashof operations as described in Section $E_k$.16.6, converges to $H_{d,p+q}(\overline{\gA})$, and has $E^1$-page given by
\begin{equation}\label{eqn:ga-alg-ss} 
	E^1_{n,p,q}(\overline{\gA}) = H_{n,p+q,p}\left(\overline{\gE_\infty\left(S^{1,0,0}_{\bF_2} \sigma \vee D^{3,2,2}_{\bF_2} \tau \right)}\right) \cong W_\infty(\bF_2  \{\sigma,\tau\}).
\end{equation}
This is the free graded commutative algebra on the tri-graded $\bF_2$-vector space $L$ with basis given by iterated $Q^i$'s applied to $\sigma$ and $\tau$. See \cref{fig:lgens} for a table of these generators in a range; we point out that the only generators of slope $< 3/4$ are given by $\sigma$, $Q^1(\sigma)$, and $\tau$.
\begin{figure}[h]
	\begin{tikzpicture}
	\begin{scope}
	\draw (-1,0)--(8,0);
	\draw (0,-1) -- (0,5.5);
	
	\foreach \s in {0,...,5}
	{
		\draw [dotted] (-.5,\s)--(8,\s);
		\node [fill=white] at (-.25,\s) [left] {\tiny $\s$};
	}
	
	\foreach \s in {1,...,6}
	{
		\draw [dotted] ({1.25*\s},-0.5)--({1.25*\s},5.5);
		\node [fill=white] at ({1.25*\s},-.5) {\tiny $\s$};
	}
	\node [fill=white] at (0,-.5) {\tiny 0};
	\node [fill=white] at ({1.25*1},0) {$\sigma$};
	\node [fill=white] at ({1.25*2},1) {$Q^1(\sigma)$};
	\node [fill=white] at ({1.25*2},2) {$Q^2(\sigma)$};
	\node [fill=white] at ({1.25*2},3) {$Q^3(\sigma)$};
	\node [fill=white] at ({1.25*2},4) {$Q^4(\sigma)$};
	\node [fill=white] at ({1.25*2},5) {$Q^5(\sigma)$};
	\node [fill=white] at ({1.25*4},3) {$Q^{2,1}(\sigma)$};
	\node [fill=white] at ({1.25*4},4) {$Q^{3,1}(\sigma)$};
	\node [fill=white] at ({1.25*4},5) {$Q^{4,1}(\sigma)$, $Q^{3,2}(\sigma)$};
	\node [fill=white] at ({1.25*3},2) {$\tau$};
	\node [fill=white] at ({1.25*6},5) {$Q^3(\tau)$};
	
		\draw [very thick,Mahogany,densely dotted] (0,0) -- ({1.7*4*1.25},{1.7*3}) node [right] {\tiny $d = \frac{3n}{4}$};
	
	\node [fill=white] at (-.5,-.5) {$\nicefrac{d}{n}$};
	\end{scope}
	\end{tikzpicture}
	\caption{The additive generators of $L$ in the range $n \leq 6$, $d \leq 5$, with filtration degree suppressed.}
	\label{fig:lgens}
\end{figure}

\vspace{.5em}

\noindent\textbf{Claim.} The differentials of \eqref{eqn:ga-alg-ss} have the following properties:
	\begin{enumerate}[\indent (i)]
		\item \label{enum:claim1} All differentials vanish on $Q^I(\sigma)$ for any admissible $I$.
		\item \label{enum:claim2} $d^1(\tau) = \sigma Q^1(\sigma)$. For $i \geq 1$, $d^{j}(\tau^{2^i}) = 0$ for $j<2^i$ and 
		\[d^{2^i}(\tau^{2^i}) = (Q^1(\sigma))^{2^i} Q^{2^{i-1},2^{i-2},\ldots,1}(\sigma) + \sigma^{2^i}Q^{2^i,2^{i-1},\ldots,1}(\sigma).\]
		\item \label{enum:claim3} $d^1(Q^I(\tau)) = d^2(Q^I(\tau)) = 0$ for admissible $I \neq \varnothing$ and $e(I)>2$.
	\end{enumerate}

\begin{proof}[Proof of claim] Part (i) follows from the map $\overline{\gE_\infty(S^{1,0,0}_{\bF_2}\sigma)} \to \overline{\gA}$ and naturality. 

We prove part (ii) by induction. As the cell $\tau$, of filtration 1, is attached along $\sigma Q^1(\sigma)$, of filtration 0, we have $d^1(\tau) = \sigma Q^1(\sigma)$. By the derivation property we have $d^1(\tau^2) = 2 d^1(\tau) = 0$, and we may then compute $d^2(\tau^2)$ using the fact that $\tau^2 = Q^2(\tau)$ so by Theorem $E_k$.16.8 the class $\tau^2$ survives to the $E^{2}$-page and the differential here satisfies
	\begin{align*}d^2(\tau^2) &= d^2(Q^2(\tau)) = Q^2(d^1(\tau)) 	= Q^2(\sigma Q^1(\sigma)) \\
	&= Q^1(\sigma)^3+\sigma^2 Q^{2,1}(\sigma)\end{align*}
	using the Cartan formula.
		
Suppose now that the result holds for $\tau^{2^i}$. Then the class $\tau^{2^i}$ survives until the $E^{2^i}$-page and on this page $\tau^{2^{i+1}}$ is the square of $\tau^{2^i}$, so is $Q^{2^{i+1}}(\tau^{2^i})$ (recall that $\tau$ has degree 2). It then follows from Theorem $E_k$.16.7 that $\tau^{2^{i+1}}$ survives until the $E^{2^{i+1}}$-page, and the differential here satisfies
	\begin{align*}d^{2^{i+1}}(\tau^{2^{i+1}}) &= d^{2^{i+1}}(Q^{2^{i+1}}(\tau^{2^i})) \\
	&= Q^{2^{i+1}}(d^{2^i}(\tau^{2^i})) \\
	&= Q^{2^{i+1}}((Q^1(\sigma))^{2^i} Q^{2^{i-1},2^{i-2},\ldots,1}(\sigma) + \sigma^{2^i}Q^{2^i,2^{i-1},\ldots,1}(\sigma)).\end{align*}
Let us consider the terms separately: $Q^{2^{i+1}}((Q^1(\sigma))^{2^i} Q^{2^{i-1},2^{i-2},\ldots,1}(\sigma))$ may be computed using the Cartan formula. Since $Q^s$ vanishes on elements of degree $d<s$, the only options are (a) splitting $Q^{2^{i+1}}$ into $2^i$ terms $Q^1$ and a single $Q^{2^i}$, resulting in $(Q^1(\sigma))^{2^{i+1}} Q^{2^i,2^{i-1},2^{i-2},\ldots,1}(\sigma)$ because $Q^s$ acts as squaring on elements of degree $s$, or (b) splitting $Q^{2^{i+1}}$ into $2^i-1$ terms $Q^1$, a single term $Q^2$ and a single $Q^{2^i-1}$, resulting in $2^i$ contributions $(Q^1(\sigma))^{2^{i+1}-2} Q^{2,1}(\sigma) (Q^{2^{i-1},2^{i-2},\ldots,1}(\sigma))^2$.

Similarly, for $Q^{2^{i+1}}(\sigma^{2^i}Q^{2^i,2^{i-1},\ldots,1}(\sigma))$ the only non-zero contribution to the Cartan formula splits $Q^{2^{i+1}}$ into $2^i$ terms $Q^0$, and a single $Q^{2^{i+1}}$, resulting in $\sigma^{2^{i+1}}Q^{2^{i+1},2^i,\ldots,1}(\sigma)$.
	
Part (iii) follows from Theorem $E_k$.16.8.
\end{proof}	

\vspace{.5em}

We will use this to analyse the behaviour of the cell-attachment filtration spectral sequence on $\overline{\gA}/\sigma$, using the map $q \colon \overline{\gA} \to \overline{\gA}/\sigma$. This spectral sequence converges to $H_{d, p+q}(\overline{\gA}/\sigma)$ and has $E^1$-page
\begin{equation} \label{eqn:ga-mod-ss} 
	E^1_{n,p,q}(\overline{\gA}/\sigma) = H_{n,p+q,p}\left(\overline{\gE_\infty\left(S^{1,0,0}_{\bF_2} \sigma \vee D^{3,2,2}_{\bF_2} \tau \right)}/\sigma \right).\end{equation}
%\begin{equation} \label{eqn:ga-mod-ss} 
%E^1_{n,p,q}(\overline{\gA}/\sigma) = H_{n,p+q,q}\left(\overline{\gE_\infty\left(S^{1,0,0}_{\bF_2} \sigma \vee D^{3,2,2}_{\bF_2} \tau \right)}/\sigma \right).\end{equation}

In terms of our description of $E^1_{*,*,*}(\overline{\gA})$ this is given by dividing out by the ideal $(\sigma)$. That is, it is the free graded commutative algebra with generators given by iterated $Q^i$'s applied to $\sigma$ and $\tau$, \emph{except} for the generator $\sigma$ itself. Let us write these generators as $q_*(Q^I(\sigma)) \coloneqq \{Q^I(\sigma)\}$ and so on. We shall compute differentials in this spectral sequence using the fact that they commute with the map $q_*$ of spectral sequences, and more generally using the module structure of \eqref{eqn:ga-mod-ss} over \eqref{eqn:ga-alg-ss},  
\[- \cdot - \colon E^2_{n,p,q}(\overline{\gA}) \otimes E^2_{n',p', q'}(\overline{\gA}/\sigma) \lra E^2_{n+n',p+p', q+q'}(\overline{\gA}/\sigma),\]
given by $\overline{\gA}/\sigma$ having the structure of a filtered module over the filtered ring $\overline{\gA}$.

\vspace{.5em}

\noindent\textbf{Claim. } $E^3_{n,p,q}(\overline{\gA}/\sigma)=0$  for $p+q < \frac{2(n-1)}{3}$.

\begin{proof}[Proof of claim] The first differentials we compute using naturality are
\begin{align*}
d^j(\{Q^I(\sigma)\}) &= 0 \text{ for $j \geq 1$} && \text{using \eqref{enum:claim1}}\\
%d^1(\{Q^I(x_\alpha)\}) &= 0\\
d^1(\{\tau^i\}) &= 0 \text{ for $i \geq 1$} && \text{using \eqref{enum:claim2}}\\
d^1(\{Q^I(\tau)\}) = d^2(\{Q^I(\tau)\}) &=0 \text{ for $I \neq \varnothing$ and $e(I)>2$} && \text{using \eqref{enum:claim3}}.\end{align*}

Because $\tau^{2^k}$ survives to $E^{2}(\overline{\gA})$, we can determine $d^2(\{\tau^{2k}\})$ using \eqref{enum:claim2}:
\begin{align*}d^2(\{\tau^{2k}\}) &= d^2(\tau^{2k} \cdot \{1\}) \\
&=d^2(\tau^{2k}) \cdot \{1\} + \tau^{2k} \cdot d^2(\{1\}) \\
&=(k\tau^{2k-2}(Q^1(\sigma)^3+\sigma^2Q^{2,1}(\sigma)) \cdot \{1\} \\
&=\{kQ^1(\sigma)^3\tau^{2k-2}\}.\end{align*}
In particular, $d^2(\{\tau^{4r}\}) = 0$ while $d^2(\{\tau^{4r+2}\}) = \{Q^1(\sigma)^3\tau^{4r}\}$. Using the module structure we deduce from this that for all $i$ and $r$
\begin{equation}\label{eqn:d2-even}\begin{aligned} d^2(\{Q^1(\sigma)^i \tau^{4r}\}) &= 0, \\
d^2(\{Q^1(\sigma)^i \tau^{4r+2}\}) &= \{Q^1(\sigma)^{i+3} \tau^{4r}\}.\end{aligned}\end{equation}

On the other hand, we cannot compute $d^2(\{\tau^{2k+1}\})$ by naturality, as $\tau^{2k+1}$ does not survive to $E^2(\overline{\gA})$. However, $\{\tau\}$ has internal degree 1, so $d^2(\{\tau\})$ has internal degree $-1$ and hence vanishes. To compute $d^2(\{\tau^{2k+1}\})$ for $k \geq 1$, we use the module structure:
\begin{align*}d^2(\{\tau^{2k+1}\}) &= d^2(\tau^{2k} \cdot \{\tau\}) \\
&=(k \tau^{2k-2} (Q^1(\sigma)^3+\sigma^2 Q^{2,1}(\sigma))) \cdot \{\tau\} + \tau^{2k}  \cdot d^2(\{\tau\}) \\
&= \{k Q^1(\sigma)^{3} \tau^{2k-1} \}.\end{align*}
In particular, $d^2(\{\tau^{4r+1}\}) = 0$, while $d^2(\{\tau^{4r+3}\}) = \{Q^1(\sigma)^3\tau^{4r+1}\}$. Using the module structure  we deduce from this that for all $i$ and $r$
\begin{equation}\label{eqn:d2-odd}\begin{aligned} d^2(\{Q^1(\sigma)^i \tau^{4r+1}\}) &= 0, \\
d^2(\{Q^1(\sigma)^i \tau^{4r+3}\}) &= \{Q^1(\sigma)^{i+3} \tau^{4r+1}\}.\end{aligned}\end{equation}

The conclusion of this discussion is as follows. Firstly $E^1_{*,*,*}(\overline{\gA}/\sigma)=E^2_{*,*,*}(\overline{\gA}/\sigma)$. Secondly, $E^2_{*,*,*}(\overline{\gA}/\sigma)$ is a free module over the subalgebra $W^{\neq \varnothing}$ of $W_{\infty}(\bF_2  \{\sigma,\tau\})$ generated by those $Q^I(\sigma)$ and $Q^I(\tau)$ with $I \neq \varnothing$, and the $d^2$-differential is $W^{\neq \varnothing}$-linear. As a $W^{\neq \varnothing}$-module, $E^2_{*,*,*}(\overline{\gA}/\sigma)$ is given by
\[W^{\neq \varnothing} \otimes \left(\bF_2[Q^1(\sigma),\tau],d\right),\]
with differential $d$ given by \eqref{eqn:d2-even} and \eqref{eqn:d2-odd}. Thus we have that $E^3_{*,*,*}(\overline{\gA}/\sigma)$ is a free $W^{\neq \varnothing}$-module on the generators $Q^1(\sigma)^i \tau^j$ with $i \in \{0,1,2\}$ and $j \equiv 0,1 \pmod 4$. These all are in bidegrees $(n,d)$ satisfying $d \geq \frac{2(n-1)}{3}$. As the additive generators of $W^{\neq \varnothing}$ have slope at least $\frac{3}{4}$, we conclude that $E^3_{n,p,q}(\overline{\gA}/\sigma)=0$ for $p+q < \frac{2(n-1)}{3}$.\end{proof}

Now that $E^3_{n,p,q}(\overline{\gA}/\sigma)$ vanishes in this range, the vanishing result for its $E^\infty$-page and hence for $H_{n,p+q}(\overline{\gA}/\sigma)$ follow. As explained earlier, this gives the desired vanishing range for $H_{n,p+q}(\overline{\gR}_{\bF_2}/\sigma)$.\end{proof}

\begin{remark}One can not do better in the above proof. The class $\tau^4$ survives to $E^4_{12,4,4}(\overline{\gA})$ and satisfies $d^4(\tau^4) = Q^1(\sigma)^4 Q^{2,1}(\sigma) + \sigma^4 Q^{4,2,1}(\sigma)$, so we have $d^4(\{\tau^4\}) = \{Q^1(\sigma)^4 Q^{2,1}(\sigma)\}$. But on the other hand we have
\[d^2(\{Q^{2,1}(\sigma) Q^1(\sigma) \tau^2\}) = \{Q^1(\sigma)^4 Q^{2,1}(\sigma)\}\]
by \eqref{eqn:d2-even} and the module structure, so $d^4(\{\tau^4\}) = 0 \in E^4_{12, 7, 0}(\overline{\gA}/\sigma)$. Now all higher differentials on $\{\tau^4\}$ would land in negative internal degree, so $\{\tau^4\}$ is a permanent cycle. There is no reason for it to be a boundary.
\end{remark}

\subsection{Sharpness}\label{sec:sharpnessf2} We saw in \cref{sec:f2-low-rank} the map
\[H_d(\fS_{n},\fS_{n-1};\bF_2) \lra H_d(\mr{GL}_n(\bF_2),\mr{GL}_{n-1}(\bF_2);\bF_2)\]
is an isomorphism for $(n,d) = (2,1),(4,2)$; both $Q^1(\sigma)$ and $Q^1(\sigma)^2$ do not destabilise. The class in bidegree $(n,d) = (4,2)$ exactly fails to satisfy the bound $d < \frac{2(n-1)}{3}$.

\begin{remark}This rules out some potential improvements of homological stability ranges for other automorphism groups. Suppose we have a ring $R$ with homomorphism to $\bF_2$ (e.g.\ $\bZ$), then upon picking a point $\sigma \in B\rm{GL}_1(R)$ we get a map $\gE_\infty^+(S^{1,0}_{\bF_2}\sigma) \to \bigsqcup_{n \geq 0} B\mr{GL}_n(R)$ of $E_\infty$-algebras. As a result of the above observation, the composition of the maps
\[\gE_\infty^+(S^{1,0}_{\bF_2} \sigma) \lra \bigsqcup_{n \geq 0} B\mr{GL}_n(R) \lra \bigsqcup_{n \geq 0} B\mr{GL}_n(\bF_2)\]
will induce an isomorphism on relative homology with $\bF_2$-coefficients of the stabilisation map in bidegree $(n,d) = (2,1),(4,2)$. Thus $H_d(\mr{GL}_n(R),\mr{GL}_{n-1}(R);\bF_2) \neq 0$ for $(n,d) = (2,1),(4,2)$. 

Similarly, the sequence of maps
\[\gE_\infty^+(S^{1,0}_{\bF_2} \sigma) \lra \bigsqcup_{n \geq 0} B\mr{Aut}(F_n) \lra \bigsqcup_{n \geq 0} B\mr{GL}_n(\bF_2),\]
the right map given by the action of a based homotopy automorphism of $\bigvee_n S^1$ on $H_1(\bigvee_n S^1;\bF_2)$, implies $H_d(\mr{Aut}(F_n),\mr{Aut}(F_{n-1});\bF_2) \neq 0$ for $(n,d) = (2,1),(4,2)$. \end{remark}

More recently Szymik \cite{Szymik} has shown that $H_3(\mr{GL}_6(\bF_2);\bF_2) \neq 0$, based on computer calculations and a theorem of Webb. Combining this with Theorem G of \cite{hepworth} it follows that this group must be one-dimensional and generated by $Q^1(\sigma)^3$. In the following lemma we give our own proof of the latter result.

\begin{lemma}
$H_{6,3}(\gR_{\bF_2})$ is at most 1-dimensional, spanned by $Q^1(\sigma)^3$.
\end{lemma}

\begin{proof}
The proof of \cref{thm:main-f2} provides a map of exact sequences
	\[\begin{tikzcd} H_{7,4}(\overline{\gA}/\sigma) \rar{\delta} \dar & H_{6,3}(\overline{\gA}) \rar{\sigma} \dar & H_{7,3}(\overline{\gA}) \rar \dar & 0 \\
	H_{7,4}(\overline{\gR}_{\bF_2}/\sigma)  \rar{\delta} & H_{6,3}(\overline{\gR}_{\bF_2}) \rar{\sigma} & H_{7,3}(\overline{\gR}_{\bF_2}) \rar & 0.\end{tikzcd}\]
The calculations of $\overline{\gA}/\sigma$ in that proof give us part of the top row: $H_{7,4}(\overline{\gA}/\sigma) = \bF_2$ generated by $\tau Q^1(\sigma)^2$ which the connecting homomorphism sends to $Q^1(\sigma)^3$.	
	
From the spectral sequence \eqref{eqn:ga-skel-ss} it follows that the left vertical map is surjective. From the vanishing range of \cref{cor:vanishing} we see that $H_{7,3}(\overline{\gR}_{\bF_2}) = 0$. Thus the composite 
\[H_{7,4}(\overline{\gA}/\sigma)\lra H_{7,4}(\overline{\gR}_{\bF_2}/\sigma)  \lra H_{6,3}(\overline{\gR}_{\bF_2})\]
is a surjective map from a $1$-dimensional $\bF_2$-vector space. By naturality its image is generated by $Q^1(\sigma)^3$.
\end{proof}

Using Szymik's calculation that $H_{6,3}(\gR_{\bF_2}) \neq 0$, it follows that $H_{6,3}(\gR_{\bF_2}) = \bF_2\{Q^1(\sigma)^3\}$. Note that this class must destabilise (twice), as applying $Q^2$ to the relation $\sigma Q_1(\sigma) = 0$ gives the relation $\sigma^2 Q^{2,1}(\sigma)+Q^1(\sigma)^3 = 0$.

As announced in the introduction, the above discussion allows us to resolve the remaining case $n=3$ of the Milgram--Priddy question.

\begin{proposition}\label{prop:MP}
The class $\det_3 \in  H^3(M_{3,3}(\bF_2);\bF_2)^{\mr{GL}_3(\bF_2) \times \mr{GL}_3(\bF_3)}$ lies in the image of the restriction map from $H^3(\mr{GL}_6(\bF_2);\bF_2)$.
\end{proposition}

\begin{proof}
The class $Q^1(\sigma) \in H_1(\mr{GL}_2(\bF_2);\bF_2)$ is represented by the matrix $\left[\begin{smallmatrix} 0 & 1 \\ 1 & 0\end{smallmatrix}\right]$, or equivalently by the conjugate matrix $\left[\begin{smallmatrix} 1 & 1 \\ 0 & 1\end{smallmatrix}\right]$. Therefore the class $Q^1(\sigma)^3 \in H_3(\mr{GL}_6(\bF_2);\bF_2)$ is carried on the subgroup $G_1 \subset \mr{GL}_6(\bF_2)$ of matrices of the form
	\[\begin{bmatrix} 1 & \ast & 0 & 0 & 0 & 0 \\
	0 & 1 & 0 & 0 & 0 & 0  \\
	0 & 0 & 1 & \ast & 0 & 0  \\
	0 & 0 & 0 & 1 & 0 & 0    \\
	0 & 0 & 0 & 0 & 1 & \ast  \\
	0 & 0 & 0 & 0 & 0  & 1   \\
	\end{bmatrix}.\]
By conjugating with the permutation matrix in $\fS_6 \subset \mr{GL}_6(\bF_2)$ which interchanges $2$ and $5$, we see that the subgroup $G_1$ is conjugate to the subgroup $G_2$ consisting of matrices of the form
	\[\begin{bmatrix} 1 & 0 & 0 & 0 & \ast & 0 \\
0 & 1 & 0 & 0 & 0 & \ast  \\
0 & 0 & 1 & \ast & 0 & 0  \\
0 & 0 & 0 & 1 & 0 & 0    \\
0 & 0 & 0 & 0 & 1 & 0  \\
0 & 0 & 0 & 0 & 0  & 1   \\
\end{bmatrix},\]
so that $H_3(G_2;\bF_2) \to H_3(\mr{GL}_6(\bF_2);\bF_2)$ also has image containing $Q^1(\sigma)^3$.

The inclusion $G_2 \hookrightarrow \mr{GL}_6(\bF_2)$ factors over the subgroup $M_{3,3}(\bF_2)$. Letting $x \in H^3(\mr{GL}_6(\bF_2);\bF_2)$ denote the cohomology class such that $\langle x,Q^1(\sigma)^3 \rangle = 1$, we conclude that the pullback of $x$ to $H^3(M_{3,3}(\bF_2);\bF_2)$ is a non-zero $\mr{GL}_3(\bF_2) \times \mr{GL}_3(\bF_2)$-invariant cohomology class. $H^*(M_{3,3}(\bF_2);\bF_2)$ can be identified with the free commutative $\bF_2$-algebra on the entries of a $(3 \times 3)$-matrix, the two factors of $\mr{GL}_3(\bF_2) \times \mr{GL}_3(\bF_2)$ acting by left and right multiplication on these entries. This evidently contains the determinant $\det_3$, and it is an elementary but long computation that $\det_3$ is the only $\mr{GL}_3(\bF_2) \times \mr{GL}_3(\bF_2)$-invariant cohomology class in $H^3(M_{3,3}(\bF_2);\bF_2)$. In \cref{lem:invgl3} below we give a shorter proof relying on some results in the literature.\end{proof}

\section{Homology of Steinberg modules} \label{sec:homology-steinberg} 

In \cref{sec:e1-steinberg-fp} we saw that the Steinberg module $\mr{St}(\bF_q^n)$ has the property that $\mr{St}(\bF_q^n) \otimes \bF_p$ is a projective and irreducible $\bF_p[\mr{GL}_n(\bF_q)]$-module (we stated it there only for $q = p^r$ with $q \neq 2$, but it also holds for $q=2$ \cite[Sections 2 and 3]{Humphreys}). As before, it follows that $H_*(\mr{GL}_n(\bF_q); \mr{St}(\bF_q^n) \otimes \bF_p)=0$.

What happens when we replace $\bF_p$ with $\bF_\ell$? Using the methods of this paper, together with Quillen's calculation of $H_*(\mr{GL}_n(\bF_q); \bF_\ell)$ for $\ell \neq p$, we will calculate $H_*(\mr{GL}_n(\bF_q); \mr{St}(\bF_q^n) \otimes \bF_\ell)$ for all $\ell \neq p$. %We write $t$ for the smallest number such that $q^t \equiv 1 \mod \ell$.

\begin{theorem}\label{thm:SteinbergHomology}
Let $q = p^r$, $\ell \neq p$, and $t>0$ be minimal such that $q^t \equiv 1 \mod \ell$.	There is an isomorphism of bigraded $\bF_\ell$-vector spaces
	\[\bigoplus_{n, d \geq 0} H_{d-n}(\mr{GL}_n(\bF_q); \mr{St}(\bF_q^n) \otimes \bF_\ell) \cong \Lambda[s\sigma, s\xi_1, s\xi_2, \ldots] \otimes \Gamma_{\bF_\ell}[s\eta_1, s\eta_2, \ldots]\]
with bidegrees $(n,d)$ of the generators of the right-hand side given by $|s\sigma| = (1,1)$, $|s\xi_i| = (t, 2it+1)$, and $|s\eta_i| = (t, 2it)$.
\end{theorem}

One may extract the following vanishing theorem from this calculation (and slight extensions are possible for specific $n$ and $q$). It improves by a factor of 2 a recent vanishing result of Ash--Putman--Sam \cite[Theorem 1.1]{APS} in the case of general linear groups of finite fields.

\begin{corollary}
	$H_d(\mr{GL}_n(\bF_q); \mr{St}(\bF_q^n))=0$ for $d < n-1$.
\end{corollary}

To begin the proof of \cref{thm:SteinbergHomology}, recall that the Steinberg module is the reduced top homology of the Tits building $T(\bF_q^n)$ for $\bF_q^n$:
	\[\mr{St}(\bF_q^n) \coloneqq \widetilde{H}_{n-2}(T(\bF_q^n);\bZ).\]
On the other hand, the $E_1$-Steinberg module is the reduced top homology of the split Tits building:
\[\mr{St}^{E_1}(\bF_q^n) \coloneqq \widetilde{H}_{n-2}(S^{E_1}(\bF_q^n);\bZ).\]
Here $T(\bF_q^n)$ is the geometric realisation of the nerve $\cT_\bullet(\bF_q^n)$ of the poset of proper subspaces of $\bF_q^n$ and their inclusions, whereas $S^{E_1}(\bF_q^n)$ is the geometric realisation of the nerve $S^{E_1}_\bullet(\bF_q^n)$ of the poset $\cS^{E_1}(\bF_q^n)$ of ordered pairs $(V_0,V_1)$ of complementary subspaces of $\bF_q^n$, where the $V_0$'s are ordered by inclusion and the $V_1$'s by reverse inclusion. Forgetting the $V_1$'s defines a map of posets, and so a map
\[\phi \colon S^{E_1}(\bF_q^n) \lra T(\bF_q^n).\]
This induces in particular a map $\phi_* \colon \mr{St}^{E_1}(\bF_q^n) \to \mr{St}(\bF_q^n)$. This is far from being an isomorphism, but we do have the following comparison.

\begin{lemma}
	For $\ell \neq p$ the induced map
	\[\phi_* \colon H_*(\mr{GL}_n(\bF_q); \mr{St}^{E_1}(\bF_q^n) \otimes \bF_\ell) \lra H_*(\mr{GL}(\bF_q^n); \mr{St}(\bF_q^n) \otimes \bF_\ell)\]
	is an isomorphism.
\end{lemma}
\begin{proof}
	As both the Steinberg and $E_1$-Steinberg modules are the reduced top homology of the spaces $T(\bF_q^n)$ and $S^{E_1}(\bF_q^n)$ respectively, which are each wedges of spheres, it suffices to prove that the induced map
	\[\phi \hcoker \mr{GL}_n(\bF_q) \colon S^{E_1}(\bF_q^n) \hcoker \mr{GL}_n(\bF_q) \lra T(\bF_q^n) \hcoker \mr{GL}_n(\bF_q)\]
	on homotopy orbits induces an isomorphism on homology with $\bF_\ell$-coefficients. The group $\mr{GL}(\bF_q^n)$ acts simplicially on the semi-simplicial sets $S^{E_1}_\bullet(\bF_q^n)$ and $T_\bullet(\bF_q^n)$, and taking levelwise homotopy orbits gives a map of semi-simplicial spaces. 
	
	The $p$-simplices of $\smash{S^{E_1}_\bullet(\bF_q^n)}$ are given by the set of splittings $V_0 \oplus V_1 \oplus \cdots \oplus V_{p+1}= \bF_q^n$ of $\bF_q^n$ into $p+2$ non-zero subspaces, whereas the $p$-simplices of $T_\bullet(\bF_q^n)$ is the set of proper flags
	\[0 < W_0 < W_1 < \cdots < W_{p} < \bF_q^n.\]
	The map $S^{E_1}_p(\bF_q^n) \to T_p(\bF_q^n)$ is a bijection on $\mr{GL}_n(\bF_q)$-orbits. The stabiliser of a decomposition $V_0 \oplus V_1 \oplus \cdots \oplus V_{p+1}$ is the subgroup $\mr{GL}(V_0) \times \cdots \times \mr{GL}(V_{p+1})$, which we may think of as block-diagonal matrices $BD \leq \mr{GL}_n(\bF_q)$ with certain block sizes, whereas the stabiliser of the associated flag with $W_i \coloneqq V_0 \oplus \cdots \oplus V_i$ is the associated subgroup of block-upper triangular matrices $BUT \leq \mr{GL}_n(\bF_q)$. The inclusion $i \colon BD \to BUT$ is split by the homomorphism $\rho \colon BUT \to BD$ which records the diagonal blocks. The kernel of $\rho$ is a $p$-group, so has vanishing $\bF_\ell$-homology: thus $\rho$ is a $\bF_\ell$-homology equivalence, and so $i$ is too.
	
	It follows that the semi-simplicial map $S^{E_1}_\bullet(\bF_q^n) \hcoker \mr{GL}_n(\bF_q) \to T_\bullet(\bF_q^n) \hcoker \mr{GL}_n(\bF_q)$ is a levelwise $\bF_\ell$-homology equivalence, so the map on geometric realisations is too.
\end{proof}

Using the identification in this lemma, \eqref{eqn:e1homology-ste1} says that
\[H_{n,d}(Q^{E_1}_\bL(\gR_{\bF_\ell}))= H^{E_1}_{n,d}(\gR_{\bF_\ell}) \cong H_{d-(n-1)}(\mr{GL}_n(\bF_q); \mr{St}(\bF_q^n) \otimes \bF_\ell).\]
On the other hand, Theorem $E_k$.13.7 gives a pointed weak equivalence 
\[S^{1} \wedge Q^{E_1}_\bL(\gR_{\bF_\ell}) \simeq {B}^{E_1}(\gR_{\bF_\ell})/S^{0,0}.\]
We may compute the homology of ${B}^{E_1}(\gR_{\bF_\ell})$ using the bar spectral sequence of Theorem $E_k$.14.1, which takes the form
\[E^2_{*, p, q} = \mathrm{Tor}^{p, q}_{H_{*,*}(\gR_{\bF_\ell})}(\bF_\ell[\bunit], \bF_\ell[\bunit]) \Longrightarrow H_{*,p+q}({B}^{E_1}(\gR_{\bF_\ell})).\]

At this point we invoke Quillen's calculation of the $\bF_\ell$-homology of the groups $\mr{GL}(\bF_q^n)$. Quillen shows \cite[Theorem 3]{quillenfinite} that there is a ring isomorphism
	\[H_{*,*}(\gR_{\bF_\ell}) \cong \bF_\ell[\sigma, \xi_1, \xi_2, \ldots] \otimes \Lambda[\eta_1, \eta_2, \ldots]\]
for classes of bidegrees $(n,d)$ given by $|\sigma|=(1,0)$, $|\xi_i| = (t, 2it)$, and $|\eta_i| = (t, 2it-1)$. The bar spectral sequence therefore begins with
	\[E^2_{n,p,q} = \Lambda[s\sigma, s\xi_1, s\xi_2, \ldots] \otimes \Gamma_{\bF_\ell}[s\eta_1, s\eta_2, \ldots]\]
for classes of tridegrees $(n,p,q)$ given by $|s\sigma|=(1,1,0)$, $|s\xi_i| = (t, 1, 2it)$, and $|s\eta_i| = (t, 1, 2it-1)$. As $\gR_{\bF_\ell}$ is an $E_2$-algebra (in fact an $E_\infty$-algebra), by Lemma $E_k$.14.3 the bar spectral sequence is a spectral sequence of $\bF_\ell$-algebras. As it is multiplicatively generated by classes of $p$-degree 1, it must collapse. Under the isomorphisms established above, this proves \cref{thm:SteinbergHomology}.

\appendix

\section{Some computations in modular invariant theory} 

In this appendix, we prove the following lemma, which is used in the proof of \cref{prop:MP}.

\begin{lemma}\label{lem:invgl3} $H^3(M_{3,3}(\bF_2);\bF_2)^{\mr{GL}_3(\bF_2) \times \mr{GL}_3(\bF_3)} = \bF_2\{{\det}_3\}$.
\end{lemma}

\begin{proof}We must prove that the determinant is the only non-zero homogeneous degree 3 polynomial $f$ with $\bF_2$-coefficients in the entries $x_{ij}$ of $(3 \times 3)$-matrices over $\bF_2$, which is invariant under both left and right multiplication by invertible matrices. 
	
	Suppose that $f \neq \det_3$ is invariant. We will rule out monomials which can have non-zero coefficients in $f$ on a case-by-case basis. Consider first the monomials $x_{i_1j_1}x_{i_2j_2}x_{i_3j_3}$ where $\# \{i_1,i_2,i_3\} = 3$ and $\# \{j_1,j_2,j_3\} = 3$. Since left and right multiplication by permutation matrices can be used to send $x_{ij}$ to $x_{\sigma(i)\tau(j)}$, all of these monomials have the same coefficient and by subtracting $\det_3$ if necessary we may assume this coefficient is $0$.
	
	Next suppose that $\#  \{i_1,i_2,i_3\} = 1$, then by applying permutations without loss of generality $i_1=i_2=i_3=1$. We may restrict $f$ to the matrices of the form
	\[\begin{bmatrix} * & 0 & 0 \\
	* & 0 & 0 \\
	* & 0 & 0 \end{bmatrix},\]
	to obtain a $\mr{GL}_3(\bF_2)$-invariant polynomial of degree 3 in the entries of vectors in $\bF_2^3$; no such non-zero polynomial exists by Dickson's theorem \cite{Dickson,Wilkerson} saying that the ring of invariants has generators in degrees $2^3-2^i$ for $0 \leq i \leq 2$. Hence each monomial with  $\#  \{i_1,i_2,i_3\} = 1$ must have zero coefficient. A similar argument tells us the same holds for monomials with $\# \{j_1,j_2,j_3\} = 1$.
	
	Next suppose that $\# \{i_1,i_2,i_3\} = 2$ and $\# \{j_1,j_2,j_3\} = 3$, then by applying permutations without loss of generality we may assume that $(i_3,j_3) = (3,3)$.  We may restrict $f$ to the matrices of the form
	\[\begin{bmatrix} * & * & 0 \\
	* & * & 0 \\
	0 & 0 & 1 \end{bmatrix},\]
	to obtain a $\mr{GL}_2(\bF_2)\times \mr{GL}_2(\bF_2)$-invariant polynomial of degree 2 in the entries of $(2 \times 2)$-matrices over $\bF_2$. By \cref{lem:invgl2} below this must be a multiple of $\det_2$ and hence only monomials with $\# \{i_1,i_2,i_3\} = 3$ and $\# \{j_1,j_2,j_3\} = 3$ can have non-zero coefficients, which we ruled out above. We conclude all monomials with $\# \{i_1,i_2,i_3\} = 2$ and $\# \{j_1,j_2,j_3\} = 3$ must have zero coefficient. A similar argument tells us the same holds for monomials with $\# \{i_1,i_2,i_3\} = 3$ and $\# \{j_1,j_2,j_3\} = 2$.
	
	Thus, if a monomial appears in $f$ with non-zero coefficient then it must have $\# \{i_1,i_2,i_3\} = 2$ and $\# \{j_1,j_2,j_3\} = 2$.  We may restrict $f$ to the matrices of the form
	\[\begin{bmatrix} * & * & 0 \\
	* & * & 0 \\
	0 & 0 & 0 \end{bmatrix},\]
	to obtain a $\mr{GL}_2(\bF_2)$-invariant polynomial of degree 3 in the entries of $(2 \times 2)$-matrices over $\bF_2$. By \cref{lem:invgl2} below this must be a multiple of $\mr{Sq}^1(\det_2)$, in which exactly those monomials of the form $x_{ij}^2 x_{i'j'}$ with $i' \neq i$ and $j' \neq i$ have non-zero coefficient. As before, all monomials of this form in $f$ must have the same coefficient, and thus $f$ is a multiple of 
	\[\sum_{i,j=1}^3 \sum_{i' \neq i,j' \neq j} x_{ij}^2 x_{i'j'}.\]
	But this is not invariant under left multiplication by 
	\[A = \begin{bmatrix} 1 & 1 & 0 \\
	0 & 1 & 0 \\
	0 & 0 & 1 \end{bmatrix},\]
	which replaces $x_{1i}$ by $x_{1i}+x_{2i}$ and leave the other terms fixed: after applying $A$ the term $x_{22}^2x_{33}$ will have coefficient $0$ (it arises twice, once from $x_{22}^2x_{33}$ and once from $x_{12}^2x_{33}$). 
\end{proof}

\begin{lemma}\label{lem:invgl2} We have 
	\begin{align*}H^2(M_{2,2}(\bF_2);\bF_2)^{\mr{GL}_2(\bF_2) \times \mr{GL}_2(\bF_2)} &= \bF_2\{{\det}_2\},\\
	H^3(M_{2,2}(\bF_2);\bF_2)^{\mr{GL}_2(\bF_2) \times \mr{GL}_2(\bF_2)} &= \bF_2\{\mr{Sq}^1({\det}_2)\}.\end{align*}\end{lemma}

\begin{proof}Let us write the entries of a $(2 \times 2)$-matrix as $\left[\begin{smallmatrix}
	a & b \\
	c & d
	\end{smallmatrix}\right]$. It is clear that the polynomials
	\[{\det}_2 = ad+bc \qquad \Sq^1({\det}_2) = a^2d+ad^2+b^2c+bc^2\]
	are invariant under left and right multiplication by invertible matrices. To see these are the only invariants we use the results of Krathwohl \cite{Krahtwohl} on Dickson invariants of pairs of vectors, that is, on $H^*(M_{2,2}(\bF_2);\bF_2)^{\mr{GL}_2(\bF_2)}$, with $\mr{GL}_2(\bF)$ acting by right multiplication. In degree 2, all such invariants are linear combinations of
	\[a^2+ab+b^2, \qquad c^2+cd+d^2, \qquad {\det}_2,\]
	and in degree 3, all such invariants are linear combinations of
	\[a^2b+ab^2, \qquad c^2d+cd^2, \qquad ad^2+bc^2, \qquad a^2d+b^2c.\]
	By using permutation matrices in $\mr{GL}_2(\bF_2)$ acting by left multiplication, we see that the $\mr{GL}_2(\bF) \times \mr{GL}_2(\bF_2)$-invariants in degree 2 must be of the form
	\[\lambda(a^2+ab+b^2 + c^2+cd+d^2)+\mu({\det}_2).\]
	If there is a non-zero invariant different from $\det_2$ we may assume it is $a^2+ab+b^2 + c^2+cd+d^2$, but this is not invariant under left multiplication by
	\[A = \begin{bmatrix} 1 & 1 \\
	0 & 1\end{bmatrix},\]
	which replaces $a$ by $a+c$ and $b$ by $b+d$ and leaves the other terms fixed: the term $d^2$ will have coefficient $0$ (it arises twice, once from $a^2$ and once from $d^2$).
	
	Similarly, the $\mr{GL}_2(\bF) \times \mr{GL}_2(\bF_2)$-invariants in degree 3 must be of the form
	\[\lambda(a^2b+ab^2 + c^2d+cd^2)+\mu(\mr{Sq}^1({\det}_2)).\]
	If there is a non-zero invariant different from $\mr{Sq}^1(\det_2)$ we may assume it is $a^2b+ab^2 + c^2d+cd^2$, but this is also not invariant under left multiplication by $A$: the term $d^2b$ will have non-zero coefficient (it arises once from $a^2b$).\end{proof}

\begin{remark}
	Alternatively, one replaces the use of \cite{Krahtwohl} by \cite[\S 4]{SmithStong}, which computes the invariants of $\mr{GL}_2(\bF_2)$ acting on $H^*(M_{2,2}(\bF_2);\bF_2)$ by $A \cdot M = A M A^\mr{t}$.
\end{remark}

%\printbibliography
\bibliographystyle{amsalpha}
\bibliography{biblio}

\vspace{\baselineskip}

\end{document}